\documentclass[12pt,reqno,a4paper]{amsart}
\usepackage{amsaddr}

 \usepackage[english]{babel}
\usepackage[left=2cm,right=2cm,top=2cm,bottom=2cm]{geometry}
\usepackage{amsmath,mathrsfs,stmaryrd}
\usepackage{amsthm}
\usepackage{mathtools}
\usepackage{empheq}
\usepackage{enumerate,enumitem}
\usepackage{dsfont}
\usepackage[hidelinks]{hyperref}
\usepackage{csquotes}
\usepackage{xcolor}
\usepackage{times}

\usepackage{cancel}
\usepackage{cleveref}
\usepackage{verbatim}

\makeatletter
\@addtoreset{equation}{section}
\makeatother

\theoremstyle{plain}
\newtheorem{theorem}{Theorem}[section] 
\newtheorem{proposition}[theorem]{Proposition}
\newtheorem{corollary}[theorem]{Corollary}

\newtheorem{lemma}[theorem]{Lemma}

\theoremstyle{definition}
\newtheorem{remark}[theorem]{Remark}

\newtheorem{definition}[theorem]{Definition}

\newcommand{\R}{\ensuremath{\mathbb R}} 

\newcommand{\dd}{\ensuremath{{\hspace{0.01em}\mathrm d}}}

\usepackage{amssymb,graphicx}

\definecolor{cadmiumgreen}{rgb}{0.0, 0.42, 0.24}
\definecolor{tangelo}{rgb}{0.98, 0.3, 0.0}
\definecolor{shamrockgreen}{rgb}{0.0, 0.62, 0.38}
\definecolor{darkolivegreen}{rgb}{0.33, 0.42, 0.18}
\definecolor{deepmagenta}{rgb}{0.8, 0.0, 0.8}
\definecolor{burgundy}{rgb}{0.5, 0.0, 0.13}
\definecolor{darkbyzantium}{rgb}{0.36, 0.22, 0.33}

\usepackage{lipsum}

\newcommand\blfootnote[1]{%
	\begingroup
	\renewcommand\thefootnote{}\footnote{#1}%
	\addtocounter{footnote}{-1}%
	\endgroup
}

\title[Statistical solutions to the SME in 1D]{Statistical solutions to the Schr\"odinger map equation in 1D, via the randomly forced Landau-Lifschitz-Gilbert equation.} 

\author[E. Gussetti, M. Hofmanová]{Emanuela Gussetti \& Martina Hofmanová}
\address{Bielefeld University, Germany }
\email{ emanuela.gussetti@uni-bielefeld.de,hofmanova@math.uni-bielefeld.de} 

\begin{document}


	\date{}


	\smallskip
	
	\begin{abstract}
		We prove the existence of statistically stationary solutions to the Schr\"odinger map equation on a one-dimensional domain, with null Neumann boundary conditions. The model is also known as Landau-Lifschitz equation.
		We deal directly with the equation in its real-valued formulation,
		\begin{align*}
		u_{t}=u_0+\int_{0}^{t} u_r\times \partial_x^2 u_r \dd r\, ,\quad |u_t|_{\mathbb{R}^3}=1\, ,
		\end{align*}
		for $t\geq 0$, without using any transform.
		To approximate the Schr\"odinger map equation, we employ the stochastic Landau-Lifschitz-Gilbert equation.
		By a limiting procedure à la Kuksin, we establish existence of a random initial datum, whose distribution is preserved under the dynamic of the deterministic equation. Among other properties, the corresponding statistically stationary solution $u$ is proved to exhibit non-trivial dynamics in space and time and to be genuinely random.
		
		With an analogous argument, we prove the existence of stationary solutions to the stochastic \sloppy Schr\"odinger map equation
		\begin{align*}
		u_{t}=u_0+\int_{0}^{t} u_r\times \partial_x^2 u_r \dd r+\int_{0}^{t}u_r\times \circ \dd W_r\, ,\quad |u_t|_{\mathbb{R}^3}=1\, ,
		\end{align*}
		for $t\geq 0$, where $W$ is an $\mathbb{R}^3$-valued Brownian motion in time and the stochastic integral is interpreted in the Stratonovich sense.
		 We discuss the relationship between the statistically stationary solutions to the Schr\"odinger map equation, the binormal curvature flow and the cubic non-linear Schr\"odinger equation.  Additionally, we prove the existence of statistically stationary solutions to the binormal curvature flow, given for $t\geq 0$ by
		 	\begin{align*}
		 v_t= v_0+\int_0^t \partial_x v_r \times \partial^2_x v_r \dd r\, , \quad |\partial_x v_t|_{\mathbb{R}^3}=1\, ,
		 	\end{align*}
	\end{abstract}
	\maketitle

	\section{Introduction}
	We explore the question: \textit{are there non-trivial statistically stationary solutions to the deterministic Schr\"odinger map equation  (SME) on a one-dimensional domain with null Neumann boundary conditions?} 
	A statistically stationary solution is a stochastic process $z=(z_t)_t$, where the trajectories $t\mapsto z_t(\omega)$ are solutions to the deterministic Schr\"odinger map equation  
	\begin{align}\label{eq:schroedinger_eq}
	z_{t}=z_0+\int_{0}^{t} z_r\times \partial_x^2 z_r \dd r \, , \quad |z_t|_{\mathbb{R}^3}=1\, ,\quad \mathrm{for}\,\, t\geq 0\, ,
	\end{align}
	for a.e.~$\omega\in \Omega$. Additionally, the distribution of these trajectories remains invariant over time, meaning that the map $\omega \mapsto z_t(\omega)$ is distributed as a fixed probability measure $\mu$ for all $t\geq 0$. 
	
	We remark that every point $Q\in \mathbb{S}^2$ is a steady state, that is a time independent solution to the SME. This implies that the measure $\mu_q=\delta_Q$, a Dirac delta in $Q$, is a statistically stationary solution to the SME. 
	However this solution is trivial, in the sense that it is constant in both $\omega$ and the spatial variable $x$, as well as in $t$.
	We are interested in observing more complex dynamics. 
	
	In a similar spirit,  we can ask whether there are configurations that remain invariant under the stochastic Schr\"odinger map equation (SSME)
	\begin{align*}
	z_{t}=z_0+\int_{0}^{t} z_r\times \partial_x^2 z_r \dd r+\int_{0}^{t}z_r\times \circ\dd W_r \, , \quad |z_t|_{\mathbb{R}^3}=1\, ,\quad \mathrm{for} t\geq 0\, ,
	\end{align*}
	 where $W$ is a $\mathbb{R}^3$-valued Brownian motion and the stochastic integral is interpreted in the Stratonovich sense. The question can be rephrased as: \textit{are there stationary solutions to the (SSME) on a one-dimensional domain with null Neumann boundary conditions?} 
	We affirmatively answer these questions and give some qualitative properties of the statistically stationary solutions. As an immediate consequence, we establish the existence of statistically stationary solutions to the binormal curvature flow.
	
	\subsection{The first main result: existence of non-trivial statistically stationary solutions to the SME}
	The proof is based on a strategy introduced by S.~Kuksin \cite{kuksin_2003} for constructing statistically stationary solutions to the Euler equation. 
	In his work, S.~Kuksin constructs statistically stationary solutions by using a sequence of stationary solutions to a properly scaled stochastic Navier-Stokes equations with additive noise. 
	 Following a similar approach, we construct statistically stationary solutions to the SME by utilizing the stochastic Landau-Lifshitz-Gilbert (LLG) equation on a bounded one-dimensional domain $D\subset \mathbb{R}$, given for $t\geq 0$ by
	 \begin{align}\label{LLG_nu}
	 u^\nu_{t}=u^\nu_{0}+\int_{0}^{t} u^\nu_r\times \partial_x^2 u^\nu_r\dd r-\nu \int_{0}^{t} u^\nu_r\times [u^\nu_r\times \partial_x^2 u^\nu_r]\dd r+ \sqrt{\nu}\int_{0}^{t} h u^\nu_r\times \circ \dd W_r\, ,\, |u^\nu_t|_{\mathbb{R}^3}=1\, ,
	 \end{align}
	 The Landau-Lifshitz-Gilbert equation is a natural candidate for this construction, as it preserves the sphere $\mathbb{S}^2$ of $\mathbb{R}^3$.  The noise is a particular type of multiplicative, instead of an additive one: also this choice preserves the sphere $\mathbb{S}^2$ of $\mathbb{R}^3$ and respects the geometry of the problem. 
	
	For every $\nu\in(0,1]$, there exists an invariant measure $\mu^\nu$ to \eqref{LLG_nu} on $H^1(\mathbb{S}^2)$ endowed with the Borel $\sigma $-algebra $\mathcal{B}_{H^1(\mathbb{S}^2)}$, and an associated stationary solution $z^\nu$ (see E.~G. \cite{LLG_inv_measure}). 
	Using It\^o's formula, every stationary solution satisfies
	\begin{align}\label{eq:rel_partial_x_h_intro}
	t\nu\mathbb{E}\left[\|z^\nu\times \partial^2_x z^\nu\|^2_{L^2}\right]=t \nu \|\partial_x h\|^2_{L^2}\, .
	\end{align}
	Equality \eqref{eq:rel_partial_x_h_intro}, together with the geometric identity
	\begin{align*}
		\|\partial_x^2 z^\nu\|_{L^2}^2=\|z^\nu\times \partial^2_x z^\nu\|_{L^2}^2+\|\partial_x z^\nu\|^4_{L^4}\, 
	\end{align*}
	leads to a uniform bound in $\nu$ for the sequence $(z^\nu)_{\nu}$ in the space
	\begin{align*}
	\mathcal{L}^2(\Omega;L^p(0,T;H^1))\cap \mathcal{L}^{1/2}(\Omega; L^p(0,T;H^2))\, ,
	\end{align*}
	for all $p\in [1,+\infty]$ and all $T>0$.  A uniform bound in $\nu$ for the sequence $(z^\nu)_{\nu}$ holds also in $\mathcal{L}^1(\Omega; C^\alpha([0,T];L^2))$, for $\alpha\in (1/3,1/2)$. A standard stochastic compactness procedure via the Skorokhod-Jakubowski theorem yields the existence of a limit $Z$ on another probability space, whose trajectories are global in time solutions to the deterministic SME. \\
	In particular, with the above procedure the constructed solutions are always genuinely random, namely they are not constant with respect to the $\omega$ variable. Concerning the dynamic in the space variable, we obtain two classes of solutions. A trivial class of solutions: in the case $h\neq 0$ and $\partial_x h=0$, the statistically stationary solution $Z$ is a random variable (i.e. the map $\omega\mapsto Z(\omega)$ is not constant) and all the trajectories $Z(\omega)$ are constant in space (i.e. $\partial_x Z= 0$). A non-trivial class of solutions: if $\partial_x h\neq 0$, there exists a statistically stationary solution $Z$ to the SME such that $Z$ is a random variable non-constant in the space variable $x$, i.e. $\partial_x Z, \partial_x^2 Z\neq 0$.
	This procedure leads to the first result of the paper. 	
	\begin{theorem}\label{th:existence_statistical_solutions}
		Let $h\in W^{1,\infty}(D;\mathbb{R})$,
		\begin{itemize}
		\item[a.] There exists a stationary stochastic process $Z$ defined on a probability space\\ $(\hat{\Omega},\hat{\mathcal{F}},\hat{\mathbb{P}})$, such that 
		$\hat{\mathbb{P}}$-a.s.~the trajectories $Z(\hat{\omega})$ are strong global solutions to \eqref{eq:schroedinger_eq} and 
		\begin{align*}
		Z(\hat{\omega})\in L^\infty([0,\infty); H^1)\cap L_{\mathrm{loc}}^2([0,\infty);H^2)\cap C([0,\infty);H^1)\, .
		\end{align*}
		\item[b.] For every $t\geq 0$ and for every $p\in [1,\infty)$
		\begin{align*}
			\hat{\mathbb{E}}[\|Z_t\times \partial_x^2 Z_t\|^2_{L^2}]\leq \|\partial_x h\|^2_{L^2}\, ,
		\end{align*}
		\begin{align*}
		\hat{\mathbb{E}}[\| \partial_x^2 Z_t\|^2_{L^2}]+\hat{\mathbb{E}}[\|\partial_x Z_t\|^4_{L^4}]+\hat{\mathbb{E}}[\|\partial_x Z_t\|^p_{L^2}]\lesssim_D \|\partial_x h\|^2_{L^2}+1\, ,
		\end{align*}
		where $\hat{\mathbb{E}}[\, \cdot\, ]$ denotes the expectation with respect to $\hat{\mathbb{P}}$.
		\item[c.]	For every $t\geq 0$ and $\hat{\mathbb{P}}$-a.s., we have
		\begin{align*}
		 \|\partial_x Z_t\|_{L^2}^2= \|\partial_x Z_0\|_{L^2}^2
		\end{align*}
		and for every $Q\in \mathbb{S}^2$ it holds
		\begin{align*}
			\|Z_t-Q\|_{L^2}^2= \|Z_0-Q\|_{L^2}^2\, .
		\end{align*}
		The map $t\mapsto \langle Z_t \rangle$ is constant in time, where $\langle\,  \cdot \, \rangle:= \int_{D} \, \cdot \, \dd x/|D|$.
		\end{itemize}
	\end{theorem}
We remark that in Theorem \ref{th:existence_statistical_solutions} b., both relations present inequalities: this means that, a priori, the derivatives of the solution could be $0$ and the solutions could be trivial. This is an additional difficulty in comparison to other results of this type: we cannot exclude the non-triviality of the solutions neither from the equation for $\|z^\nu\|^2_{L^2}$, which remains constant due to the spherical constraint, nor from the equation for $\|\partial_x z^\nu \|^2_{L^2}$.  When addressing, for instance, the non-triviality of the statistical solutions to the Euler equations via the stochastically forced Navier-Stokes equations, it is possible to conclude the non-triviality of the solutions by looking at the energy of the approximants in $L^2$ \cite{kuksin_2003}. The heart of the matter is that the conservation law $\|Z_t\|_{L^2} =\|Z_0\|_{L^2}$ is also a conservation law for the stochastic LLG equation: thus no regularity relation can be obtained. The conservation law $\|\partial_x Z_t\|_{L^2} =\|\partial_x Z_0\|_{L^2}$ is not conserved for the stochastic LLG, but we do not have enough regularity to obtain an equality in the limit in Theorem \ref{th:existence_statistical_solutions} b. 

\textcolor{black}{To tackle this issue and to show the non-triviality of $Z$, we have to find a further conservation law of the SME, which is not conserved by the stochastic LLG equation: we consider the conservation law $\langle Z_t \rangle=\langle Z_0 \rangle$. This conservation law allows us to conclude that the statistically stationary solutions constructed above are not constant in $\omega$ and, for non-constant intensity of the noise $h$, they are non-trivial in space and time: for a rigorous proof, we refer to} Section \ref{sec:proof_theorem_4_stochasticity}.
\begin{theorem}\label{th:intro_solution_stochastic}
	\begin{itemize} For $h\neq0$, the statistically stationary solutions $Z$ constructed in Theorem \ref{th:existence_statistical_solutions}:
		\item[a.] exhibit spatially non-trivial dynamics, i.e.~$\partial_x Z, \partial_x^2 Z\neq 0$ on a set of positive probability, if and only if $\partial_x h \neq 0$,
		\item[b.] exhibit non-trivial dynamics in time, i.e.~$t\mapsto Z_t$ is not a constant function on a set of positive probability, if and only if $\partial_x h \neq 0$. 
		\item[c.]  are genuinely random, i.e. the map $\hat{\omega}\mapsto Z(\hat{\omega})$ is not a constant function,
		\item[d.] if $\partial_x h \neq 0$, we cannot represent $Z$ as
		\begin{align*}
			Z_t(\hat{\omega})= \sum_{j\in \mathbb{N}}g^j_t 1_{\hat{\Omega}_j}(\hat{\omega})\, ,\quad t\geq 0\, ,\quad\hat{\mathbb{P}}\mathrm{-a.s.}\, ,
		\end{align*}
		where $g^j$ are solutions to the SME and $\{\hat{\Omega}_j\}_j$ is a partition of $\hat{\Omega}$ such that $\hat{\Omega}=\dot{\bigcup} \hat{\Omega}_j$.		
		\item[e.] If $\partial_x h\neq0$, we cannot represent $Z$ as
		\begin{align*}
			Z_t(x,\hat{\omega})=\sum_{j=1}^n g_t^j(\hat{\omega}) 1_{A_j}(x)\, ,\quad t\geq 0\, ,\quad\hat{\mathbb{P}}\mathrm{-a.s.}\, ,
		\end{align*}
		for Borel measurable sets $A_j\subset D$ disjoint and constituting a partition of $D$ and for $x\mapsto g^j_t(x,\hat{\omega})$ constant $\hat{\mathbb{P}}$-a.s.~in $A_j$.
	\end{itemize}
\end{theorem}
\begin{remark}
In Theorem \ref{th:intro_solution_stochastic} a.~and b., we can rule out the fact that the measure $\mu$ associated to the statistical solution $Z$ on $(H^2(\mathbb{S}^2),\mathcal{B}_{H^2(\mathbb{S}^2)})$ is supported exclusively on the set of steady states to the SME, which in this case coincides with the set of constant functions $f(x)=Q\in\mathbb{S}^2$ for all $x\in D$.
 In strong contrast to this case, the problem of establishing non-trivial dynamics in time of the statistical solutions remains open for other models: in J.~Bedrossian, M.~Coti-Zelati, N.~Glatt-Holtz \cite{bedrossiam_coti_zelati_holtz} for passive scalar equations, in N.~Glatt-Holtz, V.~\v{S}verák, V.~Vicol \cite{holtz_sverak} and in M.~Latocca \cite[Section 1.2.3]{latocca} for the Euler equations and in J.~Fl\"odes, M.~Sy \cite{flodes_sy} for the SQG equations. 
\end{remark}
We introduce $\Gamma:=\{\hat{\omega}\in \hat{\Omega}: \|\partial_x Z_t(\hat{\omega})\|^2_{L^2}=0  \quad\forall t\geq 0\}$, the space of spatially trivial trajectories of the statistical solution $Z$ constructed in Theorem \ref{th:existence_statistical_solutions}. Assuming that $\partial_x h\neq 0$, from Theorem \ref{th:intro_solution_stochastic}, it holds that $\hat{\mathbb{P}}(\Gamma^C)>0$. The inequality \eqref{eq:estimate_below} in Section \ref{sec:lower_bound} gives a quantitative lower bound on $\hat{\mathbb{P}}(\Gamma^C)$ in terms of $h$ and of the dimension of the domain. As a concrete example, by choosing $h(x):=0.1 \cos(x)$ on the domain $D=[0, 2\pi]$, the lower bound reads
\begin{align*}
\hat{\mathbb{P}}(\hat{\Gamma}^C) > 0.2298\, ,
\end{align*}
which tells us that more than $1/5$ of the trajectories of the statistically stationary solution $Z$ exhibit non-trivial dynamics in space and time. 
The proof of the lower bound \eqref{eq:estimate_below}  relies on a small modification of Lemma \ref{lemma:condition_for_non_trivial}. As a further consequence, Lemma \ref{lemma:condition_for_non_trivial} implies that each stationary solution $u^\nu$ to the stochastic LLG \eqref{LLG_nu} has space and time dynamic trajectories with probability one, provided $\partial_x h\neq0$. 
\begin{corollary}
	Let $h\in W^{1,\infty}(D;\mathbb{R})$ so that $\partial_x h\neq0$ and $z^\nu$ be a stationary solution to \eqref{LLG_nu} with $h$ as space component of the noise. For every fixed $\nu\in(0,1]$, 
	\begin{align*}
	\mathbb{P}(\{\omega \in \Omega: \|\partial_x z^\nu_t(\omega)\|^2_{L^2}>0\quad \forall t\geq 0\})=1\, .
	\end{align*}
\end{corollary}
The proof of this result is contained in Corollary \ref{cor:inv_measure_LLG}.

\subsection{The second main result: existence of stationary solutions to the stochastic SME}
With an analogous procedure, it is possible to show that the stochastic Schr\"odinger map equation from a one-dimensional interval $D$ to $\mathbb{S}^2$ 
\begin{align}\label{eq:stoch_schroedinger_eq}
 z_t=z_0+\int_{0}^{t} z_r\times \partial_x^2 z_r \dd r+\int_{0}^{t}z_r\times \circ \dd W_r\, ,\quad|z_t|=1\, , \quad \mathrm{for}\,\, t\geq 0 \, ,
\end{align}
 admits stationary solutions. Here $W$ is a three-dimensional Brownian motion in time with constant space component. To show this fact, it is enough to consider the modified stochastic LLG
\begin{align}\label{mod_LLG_nu}
 u^\nu_t=u^\nu_0+\int_{0}^{t} u^\nu_r\times \partial_x^2 u^\nu_r\dd r-\nu \int_{0}^{t} u^\nu_r\times [u^\nu_r\times \partial_x^2 u^\nu_r]\dd r+ \int_{0}^{t} (\sqrt{\nu}h+1) u^\nu_r\times \circ \dd W_r\, ,
\end{align}
with the spherical constraint $|u^\nu_t|=1$, for $t\geq 0$ and where $h$ can also depend on space. 
The map $h$ actually has to depend on space, to allow for statistically stationary solutions which do not coincide with the stationary spherical Brownian motion. By spherical Brownian motion we mean the unique  the $\mathbb{S}^2$-valued solution $y$ to the stochastic differential equation
\begin{align}\label{eq:intro_SBM}
	y_t=y^0+\int_{0}^{t} y_r\times \circ \dd W_r\, ,\quad\quad \mathrm{for}\,\, t\geq 0 \, ,
\end{align}
where $y^0\in \mathbb{S}^2$ (see the Appendix in \cite{LLG_inv_measure} for some properties of this stochastic process and references). The stationary spherical Brownian motion is a stationary solution to \eqref{eq:intro_SBM}. Note that the spherical Brownian motion is a space independent solution to the SSME. 
We state the second main result of the paper: for more details on the notion of martingale solutions and stationary martingale solutions, we refer to Section \ref{sec:stationary_solutions}.
\begin{theorem}\label{th:existence_stationary}
Let $h\in W^{1,\infty}(D;\mathbb{R})$. Then
\begin{itemize}
	\item[a.] There exists a stationary martingale solution $(\hat{\Omega},\hat{\mathcal{F}},(\hat{\mathcal{F}}_t)_t,\hat{\mathbb{P}}, B, Y)$ where $B$ is a $(\hat{\mathcal{F}}_t)_t$-Brownian motion and $Y$ is a stochastic process on the filtered probability space $(\hat{\Omega},\hat{\mathcal{F}},(\hat{\mathcal{F}}_t)_t,\hat{\mathbb{P}})$ such that $\hat{\mathbb{P}}$-a.s. 
	\begin{align*}
	Y(\hat{\omega})\in L^\infty([0,\infty);H^1)\cap L_{\mathrm{loc}}^2([0,\infty);H^2)\cap C([0,\infty);H^1)\, 
	\end{align*}
	and the stochastic process  $Y$ satisfies the SSME equation \eqref{eq:stoch_schroedinger_eq} (with $W$ replaced by $B$).	
	\item[b.] For every $t\geq 0$ and for every $p\in[ 1, +\infty)$
	\begin{align*}
	\hat{\mathbb{E}}[\|Y_t\times \partial_x^2 Y_t\|^2_{L^2}]\lesssim \|\partial_x h\|^2_{L^2}\, ,
	\end{align*}
	\begin{align*}
	\hat{\mathbb{E}}[\| \partial_x^2 Y_t\|^2_{L^2}]+\hat{\mathbb{E}}[\|\partial_x Y_t\|^4_{L^4}]+\hat{\mathbb{E}}[\|\partial_x Y_t\|^p_{L^2}]\lesssim_D \|\partial_x h\|^2_{L^2}+1\, ,
	\end{align*}
	where $\hat{\mathbb{E}}[\, \cdot\, ]$ denotes the expectation with respect to $\hat{\mathbb{P}}$.
	\item[c.]	For every $t\geq 0$ and $\hat{\mathbb{P}}$-a.s., it holds
	\begin{align*}
	\|\partial_x Y_t\|_{L^2}^2= \|\partial_x Y_0\|_{L^2}^2\, ,\;|\langle Y_t \rangle| =|\langle Y_0\rangle| \, .
	\end{align*}
	\item[d.] 	The stationary solution $Y$ exhibits spatially non-trivial dynamics, i.e. $\partial_x Y, \partial_x^2 Y \neq 0$, if and only if $\partial_x h\neq 0$. 
	\item[e.] The stationary spherical Brownian motion is a stationary solution to the SSME. In particular, stationary solutions to the SSME are not unique.
\end{itemize}
\end{theorem}
Up to the knowledge of the authors, there are no results concerning the stochastic Schr\"odinger map equation in the form \eqref{eq:stoch_schroedinger_eq} in the literature. In the current work we deal with a bounded one-dimensional domain and we can prove existence of stationary strong solutions, whose dynamic is non-trivial (namely there are other solutions besides the spherical Brownian motion).

 Notice that in the case of the stochastic LLG equation, the only stationary solution, if the noise is space independent, is the stationary spherical Brownian motion (see E.~G.~\cite{LLG_inv_measure}). In contrast, the stochastic SME has solutions with non-zero gradient also in presence of a noise which is space independent: this is due to the conservation of the gradient of the stochastic SME.
 \subsection{The third main result: existence of non-trivial statistically stationary solutions to the binormal curvature flow.}
 
 As discussed in R.~L.~Jerrard, D.~Smets \cite{jerrard_smets_T_1_S_2}, the Schr\"odinger map equation is related to the binormal curvature flow equation and to the cubic non-linear Schr\"odinger equation by means of some geometric transformations.
 We address and discuss the following questions:
 \begin{enumerate}
 	\item Can we apply rigorously those transforms to the trajectories of the statistically stationary solutions?
 	\item Are the transformed processes also statistically stationary solutions to the corresponding equations?
 \end{enumerate}
We apply to each trajectory $Z(\omega)$ the transformation 
\begin{align*}
X_t(\omega):=f(Z_t(\omega))\, ,\quad\quad\mathrm{where}\quad f(v):=\int_{0}^{\cdot} v(x)\dd x\, .
\end{align*}
The trajectory $X(\omega)$ is a solution to the binormal curvature flow (BCF), namely its time evolution is described for all $t\geq 0$
\begin{align*}
X_t(\omega)= X_0(\omega)+\int_0^t \partial_x X_r (\omega)\times \partial^2_x X_r (\omega)\dd r\, .
\end{align*}
A natural question arises: \textit{is $X$ also a statistically stationary solution to the BCF?} Since the transformation $f$ is time independent, the answer is positive.
\begin{theorem}\label{th:statisticalBCF}
	Let $Z$ be a statistically stationary solution to the SME on a probability space $(\hat{\Omega},\hat{\mathcal{F}}, \hat{\mathbb{P}})$ constructed as in Theorem \ref{th:existence_statistical_solutions} and let $X_t(\hat{\omega}):=f(Z_t(\hat{\omega}))$ for a.e. $\hat{\omega}\in \hat{\Omega}$ and for all $t\geq 0$. Then it holds
	\begin{itemize}
		\item[a.]  $X(\hat{\omega})\in C([0,+\infty);H^2)$ $\hat{\mathbb{P}}$-a.s. and $X(\hat{\omega})$ is $\hat{\mathbb{P}}$-a.s. a solution to the binormal curvature flow.
		\item[b.]  $X=(X_t)_t$ is a stationary process.
		\item[c.] The solution $X$ has the following properties:
		\begin{itemize}
			\item[i.] it exhibits spatially non-trivial dynamics, i.e.~$\partial_x X,\partial_x^2 X\neq 0$ if and only if $\partial_x h\neq 0$,
			\item[ii.] it is genuinely random, i.e.~$\hat{\omega}\mapsto X(\hat{\omega})$ is not a constant function,
			\item[iii.] it exhibits non-trivial dynamics in time, i.e.~$t\mapsto X_t$ is not a constant function if and only if $\partial_x h\neq 0$.
		\end{itemize}
	\end{itemize}
\end{theorem}
The binormal curvature flow has a physical interpretation: it describes the evolution in time of a vortex filament in a 3D fluid. The model has been formally derived by L.~S.~Da Rios \cite{da_rios}. We mention R.~L.~Jerrard, C.~Seis \cite{jerrard_seis} for a rigorous derivation. For more results and other properties of the solutions to the BCF we mention the works of V.~Banica and L.~Vega \cite{banica_vega_1,banica_vega_2,banica_vega_3,banica_vega_4}.
 
\subsubsection{Relations with the cubic focusing non-linear Schr\"odinger  equation.}
	A standard approach to study the Schr\"odinger map equation is to convert it into the cubic focusing non-linear Schr\"odinger  equation by means of the Hashimoto transform, introduced by H.~Hashimoto \cite{hashimoto}. The Hashimoto transform is a change of variable of the form
	\begin{align}\label{eq:hashimoto}
		q_t(x)=k_t(x)\exp\left(\mathrm{i}\int_0^x \tau_t(y)\dd y\right),\; k_t(x):=|\partial_x z_t(x)|\, ,\; \tau_t(x):=\frac{z_t(x)\times \partial_x z_t(x)\cdot \partial_x^2 z_t(x)}{|\partial_x z_t(x)|^2}\, ,
	\end{align}
	where $\mathrm{i}$ denotes the imaginary part, $z$ is a smooth solution to the SME and $k$  and $\tau$ are respectively the curvature and the torsion of the solution to the SME. Such transform links the SME with the cubic non-linear Schr\"odinger equation (CNSE),
	\begin{align}\label{eq:cubic_NSE}
		\mathrm{i}  p_t= \mathrm{i} p_0+\int_{0}^{t} [\partial_x^2 p_r+\frac{1}{2}p_r|p_r|^2]\dd r\, .
	\end{align}
	The transformation \eqref{eq:hashimoto} does not lead directly to the cubic non-linear Schr\"odinger equation in the form \eqref{eq:cubic_NSE}. Indeed, the actual equation satisfied by applying the transform \eqref{eq:hashimoto} is the non-local equation
	\begin{align}\label{eq:non_local_cubic_NSE}
		\mathrm{i}  q_t= \mathrm{i}q_0+\int_{0}^{t} [\partial_x^2 q_r+\frac{1}{2}q_r|q_r|^2]\dd r-\int_{0}^{t}A(r)q_r\dd r\, ,\quad A(r):=\left(2\frac{\partial^2_x q_r-k_r\tau_r^2}{k_r}+k_r^2\right)(0)\, .
	\end{align}
	Passing from the non-local equation \eqref{eq:non_local_cubic_NSE} to the CNSE, requires a time shifting: this procedure does not preserve the invariance of the law at every time and we cannot transform a statistically stationary solution to the SME into a statistically stationary solution to the CNSE. For this reason we cannot say whether the statistically stationary solutions to the  the cubic non-linear Schr\"odinger equation (NLS) in $d=1$, as constructed in  S.~Kuksin, A.~Shirikyan \cite{kuksin_shirikyan_2004} or in B.~Ferrario and M.~Zanella \cite{ferrario_zanella_stat_sol_NSE}, correspond to ours after the transform.

\subsection{Literature on statistically stationary solutions.} 
The literature on statistically stationary solutions for dispersive PDEs is vast: we mention here only the results immediately related with the current work. For a more exhaustive bibliography, we refer the interested reader to the introduction in \cite{sy_19}.

The problem of constructing statistically stationary solutions for the NSE has been considered for the first time in one dimension by J. L. Lebowitz, H. A. Rose, and E. R. Speer \cite{lebowitz_rose_speer}. In J.~Bourgain \cite{Bourgain, bourgain_94}, the author constructs a Gibbs invariant measure for the 2D defocusing cubic Schr\"odinger equation on $\mathbb{T}^2$. 
A second methodology, which we employ in this work, is the so called \textit{fluctuation-dissipation method}. This method has been introduced by S.~Kuksin \cite{kuksin_2003} in the framework of fluid-dynamics and has been developed and applied to other equations \cite{kuksin_KdV_2008,kuksin_KdV_2010}.  In relation to this work, we mention the results on existence of stationary solutions obtained by S.~Kuksin and A.~Shirikyan  \cite{kuksin_shirikyan_2004, kuksin_2008, Shirikyan_local_times}. In a recent work, B.~Ferrario and M.~Zanella \cite{ferrario_zanella_stat_sol_NSE} show existence of statistically stationary solutions to the NSE in an extended setting in comparison to the one of S.~Kuksin. 
We mention a new approach based on mixing techniques from S.~Kuksin and J.~Bourgain introduced in \cite{sy_19} by M.~Sy and M.~Sy, X.~Yu \cite{sy_xu_2021} and further employed and developed in M.~Sy, X.~Yu \cite{sy_xu_2021_2,sy_xu_2022}.

Up to the authors knowledge, the existence of statistical solutions for the SME, for the BCF and the existence of stationary solutions to the SSME are established for the first time in this work. The geometry of the problem requires a specific type of multiplicative linear noise, in contrast to the usual additive noise employed in the literature. This poses additional difficulties in the present work in establishing the non-triviality in space and time of a set of positive measure of trajectories. Moreover, we are able to obtain a quantitative non-triviality result and to study the stochastic setting.

\subsection{Organization of the paper.} In Section \ref{sec:notations_setting}, we provide some basic notations. In Section \ref{sec:LLG_known}, we recall some known results and derive new estimates for the stochastic LLG equation. In Section \ref{sec:proof_theorem_1}, we prove the existence of statistically stationary solutions to the SME as stated in Theorem \ref{th:existence_statistical_solutions}. In Section \ref{sec:proof_theorem_4_stochasticity}, we demonstrate that the statistically stationary solutions to the SME exhibit non-trivial dynamics in space, in time and are genuinely random (see Theorem \ref{th:intro_solution_stochastic}). In Section \ref{sec:lower_bound}, we give a lower bound on the probability of having non-trivial dynamics in space and time for the statistically stationary solutions to the SME.
In Section \ref{sec:stationary_solutions}, we establish Theorem \ref{th:existence_stationary}: the existence of stationary solutions to the stochastic SME. In Section \ref{sec:stat_sol_other_eq}, we discuss the relationship between the SME and other equations and we prove Theorem \ref{th:statisticalBCF}: the existence of statistically stationary solutions to the binormal curvature flow. Finally, Appendix \ref{sec:appendix_a} contains some useful inequalities; Appendix \ref{sec:appendix_b} contains some computations for the It\^o lemma for the spherical noise; Appendix \ref{sec:appendix_c} includes some known embedding theorems; Appendix \ref{sec:appendix_d} provides known results on stationary processes.

\subsection{Acknowledgements:} 
The authors gratefully acknowledge the financial support the Deutsche Forschungsgemeinschaft (DFG, German Research Foundation) – SFB 1283/2 2021 – 317210226 (Projects A1-B7).

	\section{Notations and setting.} \label{sec:notations_setting}
	\subsection{Some notations}
	For $a,b\in \R^3$, we denote by $a\cdot b$ the inner product in $\R^3$, and by $|\cdot|$ the norm inherited from it (we will not distinguish between the different dimensions, it will be clear from the context). We recall the definition of cross product
	$a\times b:=(a_2b_3-a_3b_2,a_3b_1-a_1b_3,a_1b_2-a_2b_1)$, for $a\equiv(a_1,a_2,a_3),b\equiv(b_1,b_2,b_3)\in \R^3$.
	We denote by $\mathbb{S}^2:=\{a\in\mathbb{R}^3:|a|_{\mathbb{R}^3}=1 \}$ the unit sphere in $\R ^3$.
	Let $E$ be a Banach space. For $T>0$, we denote the space of continuous functions defined on $[0,T]$ with values in $E$ by $C([0,T];E)$. We denote by $C_{\mathrm{loc}}([0,+\infty);E)$ the space of maps $f\in C([0,T];E)$ for all $T>0$. For $\alpha \in [0,1]$, we denote by $C^\alpha([0,T];E)$ the space of $\alpha$-H\"older continuous functions from $[0,T]$ with values in $E$.\\
	
	We denote by $C_{\mathrm{w}}([0,T]; E)$ the set of continuous functions $f:[0,T]\rightarrow E$ continuous with respect to the weak topology, i.e. such that the scalar functions $t\mapsto \langle g^*, f(t,\cdot) \rangle_{E^*,E} $ belong to $C([0,T];\mathbb{R})$ for any $g^*\in E^*$. We say that a sequence $(f_n)_n$ converges to $f$ in $C_{\mathrm{w}}([0,T]; E)$ provided
	\begin{align*}
	\lim_{n\rightarrow 0}\sup_{t\in [0,T]}|\langle g^*,f_n-f\rangle_{E^*,E}|=0 \quad \forall g^*\in E^*\, .
	\end{align*}
	
	Let $D\subset\mathbb{R}$ be an open bounded interval of $\mathbb{R}$. Denote by $\mathbb{N}$ the space of natural numbers and $\mathbb{N}_0:=\mathbb{N}\cup\{0\}$.
	For $n\in\mathbb{N}$, we consider the usual Lebesgue spaces $L^p:=L^p(D ;\R^n)$, for $p\in[1,+\infty]$ endowed with the norm $\|\cdot\|_{L^p}$ and the classical Sobolev spaces $W^{k,q}:=W^{k,q}(D;\R ^n)$ for integer $q\in [1,+\infty]$ and $k\in\mathbb{N}$ endowed with the norm $\|\cdot\|_{W^{k,q}}$.
	We also denote by $H^k:=W^{k,2}(D ;\R ^n)$. 
	We need to consider also functions taking values in $\mathbb{S}^2\subset\R^3$: we therefore introduce the notation
	\begin{equation*}
	H^k(\mathbb{S}^2):=H^k(D;\R^3)\cap \{g:D \rightarrow \R^3 \, \textrm{s.t.}\, |g(x)|=1\,\text{ a.e.}\, x\in D \}\, ,
	\end{equation*}
	for $k\in\mathbb{N}_0$. 
	Finally, we will denote by $L^p(W^{k,q}):=L^p([0,T];W^{k,q}(D ;\R ^n))$. We indicate with $C^k_0(D)$ the space of real valued functions with compact support on $D$, $k$-times continuously differentiable and such that every derivative is compactly supported on $D$. Let $(\Omega,\mathcal{F},\mathbb{P})$ be a probability space. We denote by $\mathcal{L}^p(\Omega;E)$ the usual Lebesgue space with respect to the probability measure $\mathbb{P}$.

	We recall the basic notion of stationary solution for a stochastic process. 	Let $X\equiv(X_t)_{t\geq 0}$ be a $E$-valued measurable stochastic process. We say that $X$ is stationary provided the laws
		\begin{align*}
		\mathcal{L}(X_{t_1},\, X_{t_2}, \dots, X_{t_n})\, ,\quad \mathcal{L}(X_{t_1+\tau},X_{t_2+\tau},\dots, X_{t_n+\tau})
		\end{align*}
		coincide on $E^{\times n}$ for all $\tau>0$  and for all $t_1,\cdots,t_n\in [0,+\infty)$.

	\section{Some results on the Landau-Lifschitz-Gilbert equation in 1D}\label{sec:LLG_known}
	We recall some known results regarding the Landau-Lifschitz-Gilbert equation on a one-dimensional domain. There are two main approaches to the study of the well posedness of the stochastic equation in one dimension: the martingale approach via the classical Stratonovich calculus (introduced in Z.~Brze\'zniak, B.~Goldys, Jegeraj \cite{brzezniak_LDP}) and the rough paths approach (explored in E.~G., A.~Hocquet \cite{LLG1D}, E.~G.~\cite{CLT} and K.~Fahim, E.~Hausenblas, D.~Mukherjee \cite{fahim_hausenblas}). We mention also the monograph from B.~Guo and X.~Pu \cite{stochastic_LL}. The solutions obtained by means of the two approaches coincide up to a set of null measure (it follows from the fact that the Stratonovich stochastic integral and its rough paths equivalent coincide up to a set of null measure, see e.g.~\cite[Theorem 5.14]{FrizHairer}). Since we do not need to use the rough paths formalism in this work, we do not introduce a rigorous notion of solution via rough paths in this setting: we refer the interested reader to \cite{LLG1D, LLG_inv_measure}.
	We rely on the classical Stratonovich calculus to establish the existence of statistically stationary solutions, thus we introduce only the notion of solution given in  Z.~Brze\'zniak, B.~Goldys, Jegeraj \cite{brzezniak_LDP}. 
	In E.~G.~\cite{LLG_inv_measure}, the existence of an invariant measure to the stochastic LLG equation is proved.
	We remark that the result of existence of invariant measures is obtained via the rough paths construction, but holds also within the formalism of martingale theory. 
	
 	There exists a unique solution to the stochastic LLG, with pathwise properties as well.
 	\begin{proposition} 
 		Let $T\geq 0$, $h\in W^{2,\infty}(\mathbb{D};\mathbb{R})$ and $u^0\in H^1(\mathbb{S}^2)$. Let $(\Omega,\mathcal{F}, (\mathcal{F}_t)_t,\mathbb{P})$ be a filtered probability space let $W$ be an $(\mathcal{F}_t)_t$-Brownian motion with values in $\mathbb{R}^3$.  There exists a unique solution $u\in L^\infty(0,T;H^1)\cap L^2([0,T]; H^2)\cap C^\alpha([0,T]; H^1)$ to 
 		\begin{equation}
 		\label{LLG}
 		u_{t}-u_0+\int_0^t\left[u_r\times(u_r\times\partial_x^2u_r) - u_r\times\partial_x^2 u_r\right] \dd r
 		=\int_{0}^{t}h u_r\times \circ \dd W_r\, 
 		\end{equation}
 		with initial condition $u^0$.
 	\end{proposition}
	 Also a pathwise higher order regularity result holds.
	 \begin{proposition}
	 	Let $k\in \mathbb{N}$, $h\in W^{k+1,\infty}$ and $u^0\in H^{k}(\mathbb{S}^2)$. There exists a unique solution to \eqref{LLG} with initial condition $u^0$ and such that $u\in L^\infty([0,T]; H^k)\cap L^2([0,T]; H^{k+1})\cap C^\alpha([0,T]; H^{k-1})$.
	 \end{proposition}
	We recall the results in E.~G. \cite{LLG_inv_measure} on existence of an invariant measure to \eqref{LLG}.
	\begin{proposition}\label{pro:a_priori_approx_nu}
		\begin{itemize}
			\item The Markov semigroup $(P_t)_t$ on $H^1(\mathbb{S}^2)$ associated to \eqref{LLG} admits an invariant measure $\mu$.
			\item Every stationary solution $z$ to \eqref{LLG} fulfils  $\mathbb{E}[\|z\|^2_{H^1}]<+\infty$ and
			\begin{align}\label{eq:equality_u_cross_laplacian}
				\mathbb{E}[\|z\times \partial_x^2 z\|^2_{L^2}]=\|\partial_x h\|^2_{L^2}\, ,\quad \mathbb{E}[\|\partial_x^2 z\|_{L^2}]\lesssim \|\partial_x h\|^2_{L^2}\, .
			\end{align}
		\end{itemize}
	\end{proposition}

	\subsection{A priori estimates and equalities on the perturbed equation.}
	In this section, we prove a priori bounds on a sequence $(z^\nu)_\nu$ of stationary solutions to  \eqref{LLG_nu}. Without loss of generality, we assume that the stationary solutions $(z^\nu)_\nu$ are all constructed on a common probability space with the same Brownian motion. We first recall the following relation on the norm of the gradient $\partial_x z^\nu$.
	\begin{lemma}\label{lemma:u_dot_nablau_0}\cite[Lemma 4.6]{LLG_inv_measure} Let $u\in L^\infty(0,T;H^1)$ such that $|u_t(x)|_{\mathbb{R}^3}=1\,$ for a.e. $(t,x)\in[0,T]\times D\,$, then
		\begin{align*}
		u_t(x)\cdot \partial_x u_t(x)=0\quad a.e.\,\, (t,x)\in[0,T]\times D\, .
		\end{align*}
		In particular, it follows that $|u_t(x)\times \partial_x u_t(x)|_{\mathbb{R}^3}=|\partial_xu_t(x)|_{\mathbb{R}^3}$ for a.e. $(t,x)\in[0,T]\times D\,$.
	\end{lemma}
		\begin{lemma}\label{lemma:a_priori_bound_derviative_p}
		For all $\nu\in(0,1]$ and for all $p\in [2,+\infty)$, and $t\geq 0$
		\begin{align*}
		\mathbb{E}[\|\partial_x z^\nu_t\|_{L^2}^p]\lesssim_p \|\partial_x h\|^{p}_{L^2}\, ,
		\end{align*}
		\begin{align}\label{eq:condition_stationarity_nu}
		\mathbb{E}\left[\sup_{t\in [0,T]}\|\partial_x z^\nu_t\|_{L^2}^p\right]\lesssim_{T,h,p} \|\partial_x h\|^{p}_{L^2}\, ,
		\end{align}
		where the implicit constants are independent on $\nu$.
	\end{lemma}
	
	\begin{proof}
		We fix $g\in C^2(\mathbb{R};\mathbb{R})$ and we describe the evolution of $g(\|\partial_x z^\nu_t\|^2_{L^2})$ by It\^o's formula as,
		\begin{align*}
		g(\|\partial_x z^\nu_t\|^2_{L^2})&= g(\|\partial_x z^\nu_s\|^2_{L^2})-2\nu\int_s^t g'(\|\partial_x z^\nu_r\|^2_{L^2}) \|z^\nu_r\times \partial_x^2 z^\nu_r\|^2_{L^2} \dd r+ 2\nu \int_s^t g'(\|\partial_x z^\nu_r\|^2_{L^2})\|\partial_x h\|^2_{L^2}\dd r \\
		&\quad+2\int_{s}^{t} g'(\|\partial_x z^\nu_r\|^2_{L^2}) \left(\partial_x z^\nu_r,\sqrt{\nu} \partial_xh z^\nu_r\times \dd W_r\right)+\frac{ 1}{2} \int_s^t g''(\|\partial_x z^\nu_r\|^2_{L^2}) \dd\langle \langle \|\partial_x z^\nu_{\cdot}\|^2_{L^2}\rangle\rangle_r\, .
		\end{align*}
		The quadratic variation $\langle \langle \|\partial_x z^\nu_t\|^2_{L^2} \rangle\rangle_t$ reads
		\begin{align*}
		\langle \langle \|\partial_x z^\nu_\cdot\|^2_{L^2} \rangle \rangle_{t} &= \langle \langle 2\int_0^{\cdot} \int_{D}\sqrt{\nu} \partial_x h z^\nu_r\times \partial_x z^\nu_r\cdot  \dd W_r \rangle \rangle_{t}\\
		&=4\nu\int_0^t \left|\int_{D} \partial_x h z^\nu_r\times \partial_x z^\nu_r\dd x\right|^2\dd r\, .
		\end{align*}
		By choosing $g(x)=x^{p/2}$ for $p\geq 2$, $g'(x)=px^{p/2-1}/2$ and $g''(x)=p/2(p/2-1)x^{p/2-2}$, and from Lemma \ref{lemma:ito_formula_square} (It\^o's formula for $p=2$)
		\begin{align*}
		&\|\partial_x z^\nu_t\|_{L^2}^p=\|\partial_x z^\nu_0\|_{L^2}^p- p\nu\int_{0}^{t} \|\partial_x z^\nu_r\|_{L^2}^{p-2} \|z^\nu_r\times \partial^2_x z^\nu_r\|_{L^2}^2  \dd r +p\nu\int_{0}^{t} \|\partial_x z^\nu_r\|_{L^2}^{p-2}\|\partial_x h\|^2_{L^2}\dd r\\
		&\quad+ p\int_{0}^{t} \langle \|\partial_x z^\nu_r\|_{L^2}^{p-2}  \partial_x z^\nu_r , \sqrt{\nu} \partial_x h   z^\nu_r  \times  \dd W_r\rangle +\frac{ \nu p(p-2)}{2}\int_0^t \|\partial_x z^\nu_r\|_{L^2}^{p-4}\left|\int_{D} \partial_x h z^\nu_r\times \partial_x z^\nu_r\dd x\right|^2 \dd r\, .
		\end{align*}
		As a consequence of the elementary equality $(a\times b)\cdot a=0$, for all $a,b\in \mathbb{R}^3$, the following integral is $0$ and does not appear in the above equality
		\begin{align*}
		p\int_{0}^{t} \langle \|\partial_x z^\nu_r\|_{L^2}^{p-2}  \partial_x z^\nu_r , \sqrt{\nu}  h   \partial_x z^\nu_r  \times  \dd W_r\rangle=0\, .
		\end{align*}
		From the stationarity of $z^\nu_r$ and by applying expectation, we obtain for all $t\geq 0$
		\begin{align}\label{eq:00}
		p\nu t\mathbb{E}[\|\partial_x z^\nu_t\|^{p-2}_{L^2}\|z^\nu_t\times \partial^2_x z^\nu_t\|_{L^2}^2 ]&=\nu p t\mathbb{E}[\|\partial_x z^\nu_t\|^{p-2}_{L^2}\|\partial_x h\|^2_{L^2}]\\
		&\quad+\frac{ \nu p(p-2)t}{2}\mathbb{E}\left[\|\partial_x z^\nu_t\|_{L^2}^{p-4}\left|\int_{D} \partial_x h z^\nu_t\times \partial_x z^\nu_t\dd x\right|^2 \right] \nonumber\, ,
		\end{align}
		where we used that the stochastic integral is a martingale. From the Cauchy-Schwarz inequality,
		\begin{align*}
			\left|\int_{D} \partial_x h z^\nu_t\times \partial_x z^\nu_t\dd x\right|^2\leq \|\partial_x h\|_{L^2}^2\|z^\nu_t\times \partial_x z^\nu_t\|_{L^2}^2=\|\partial_x h\|_{L^2}^2\|\partial_x z^\nu_t\|_{L^2}^2\, .
		\end{align*}
		From the Poincaré's inequality and Lemma \ref{lemma:u_dot_nablau_0},
		\begin{align}\label{eq:01}
		\| \partial_x z^\nu_t\|_{L^2}^2=\|z^\nu_t\times \partial_x z^\nu_t\|_{L^2}^2 \lesssim_{D} \|z^\nu_t\times \partial^2_x z^\nu_t\|_{L^2}^2 \, .
		\end{align} 
		By inserting \eqref{eq:01} into \eqref{eq:00}, for all $\epsilon>0$
		\begin{align*}
		\mathbb{E}[\|\partial_x z^\nu_t\|^{p}_{L^2}]\lesssim \mathbb{E}[\|\partial_x z^\nu_t\|^{p-2}_{L^2}\|\partial_x h\|^2_{L^2}] \lesssim_p \epsilon\mathbb{E}[\|\partial_x z^\nu_t\|_{L^2}^{p}]+C_\epsilon\|\partial_x h\|^p_{L^2}\, ,
		\end{align*}
		where we have used the weighted Young's inequality. By absorbing the term $\mathbb{E}[\|\partial_x z^\nu_t\|^{p}_{L^2}]$ for a small $\epsilon$, we obtain the uniform estimate for all $p\geq 2$
		\begin{align*}
		\mathbb{E}[\|\partial_x z^\nu_r\|^{p}_{L^2}]\lesssim_{D,p} \|\partial_x h\|^{p}_{L^2}\, .
		\end{align*}
		Finally, estimate \eqref{eq:condition_stationarity_nu} follows from the previous uniform bound and the Burkholder-Davis-Gundy inequality.
	\end{proof}
	We recall a well known fundamental equality for the LLG equation.
	\begin{lemma}\label{lemma:fund_equality}
		For every $\nu\in(0,1]$, it holds for a.e. $t\geq 0$ and $\mathbb{P}$-a.s. that
		\begin{align} \label{eq:fundamental_equality}
		\|\partial^2_x z^\nu_t\|^2_{L^2}=	\|z^\nu_t\times \partial^2_x z^\nu_t\|^2_{L^2}+\|\partial_x z^\nu_t\|^4_{L^4}\, .
		\end{align}
	\end{lemma}
	\begin{proof}
		From e.g. \cite[Lemma 2.3]{brzezniak_LDP}, for all $u\in H^2(\mathbb{S}^2)$ so that $|u(x)|_{\mathbb{R}^3}=1$  it holds for all $\phi \in L^2$ that
		\begin{align}\label{eq:u_cross_u_cross_Laplace}
		-\langle u\times (u\times \partial_x^2 u), \phi\rangle =\langle \partial_x^2 u+u|\partial_x u|^2 , \phi\rangle\, .
		\end{align}
		By applying this result to $z^\nu_t$ for a.e. $(t,\omega)\in [0,T]\times \Omega$ and recalling that $z^\nu_t\in H^2$ a.e., 
		\begin{align*}
		-\langle z^\nu_t\times (z^\nu_t\times \partial_x^2 z^\nu_t), \partial_x^2 z^\nu_t\rangle =\langle \partial_x^2 z^\nu_t+z^\nu_t|\partial_x z^\nu_t|^2 , \partial_x^2 z^\nu_t\rangle\, .
		\end{align*}
		From the elementary relation $a\times (a\times b)\cdot b=-|a\times b|^2_{\mathbb{R}^2}$ for all $a,b\in \mathbb{R}^3$ and the fact that $z^\nu\cdot \partial_x^2 z^\nu =-|\partial_x z^\nu|^2$ a.s., the claim follows.
	\end{proof}
	As a consequence of  Lemma \ref{lemma:u_dot_nablau_0}, Lemma \ref{lemma:a_priori_bound_derviative_p} and Lemma \ref{lemma:fund_equality}, we obtain the following a priori bounds.
	\begin{lemma} \label{cor:laplaciano_nu}The uniform bounds hold for all $t\geq 0$
		\begin{align*}
		\quad\sup_\nu\mathbb{E}[\|\partial_x z^\nu_t\|^4_{L^4}]\leq C\, , \quad	\sup_\nu\mathbb{E}[\|\partial^2_x z^\nu_t\|^2_{L^2}]\leq C\, ,
		\end{align*}
		where the constant $C=C(T,h)>0$ is independent on $\nu\in(0,1]$.
	\end{lemma}
\begin{proof} 
From Lemma \ref{lemma:u_dot_nablau_0}, the equality $|\partial_x z^\nu_t(x)|_{\mathbb{R}^3}=|z^\nu_t(x)\times \partial_x z^\nu_t(x)|_{\mathbb{R}^3}$ holds for a.e. $(t,x)\in [0,T]\times D$ $\mathbb{P}$-a.s. The Ladyzhenskaya's inequality in one dimension (see Lemma \ref{lemma:interp_ladyz}) implies
\begin{align*}
	\|\partial_x z^\nu_t\|^4_{L^4}=\|z^\nu_t\times \partial_x z^\nu_t\|^4_{L^4}
	&\lesssim \|\partial_x z^\nu_t\|_{L^2}^3\|z^\nu_t\times \partial^2_x z^\nu_t\|_{L^2}\\
	&\lesssim \|\partial_x z^\nu_t\|_{L^2}^6+ \|z^\nu_t\times \partial^2_x z^\nu_t\|_{L^2}^2\, .
\end{align*}
By taking expectation in \eqref{eq:fundamental_equality}, by using \eqref{eq:equality_u_cross_laplacian} (where we remark that the constant appearing is independent on $\nu$) and Lemma \ref{lemma:a_priori_bound_derviative_p},  the assertion follows.
\end{proof}
We look now at the time regularity.
\begin{lemma}\label{lemma:time_regularity}
	The sequence $(z^\nu)_{\nu}$ is uniformly bounded in $\mathcal{L}^2(\Omega;C^{\alpha}([0,T];L^2))$, for all $\alpha\in (1/3,1/2)$.
\end{lemma}
\begin{proof}
	To show that $(z^\nu)_\nu$ is uniformly bounded in $\mathcal{L}^2(\Omega;C^{\alpha}([0,T];L^2))$, we prove that the drift is in $\mathcal{L}^2(\Omega;C^{1/2-}([0,T];L^2))$ and the stochastic integral lies in $\mathcal{L}^2(\Omega;C^{\alpha}([0,T];L^2))$, for $\alpha\in (1/3,1/2)$.
	We write the stochastic LLG equation with respect to the It\^o integration theory: we refer to Lemma \ref{lemma:ito_stratonovich_correction} for the computations of the It\^o-Stratonovich correction term.
For all $s\leq t\in [0,T]$, by taking the $L^2$-norm of the drift,
\begin{align*}
	&\mathbb{E}\left\|\int_{s}^{t}\left[ z^\nu_r\times \partial_x^2 z^\nu_r-\nu z^\nu_r\times (z^\nu_r\times \partial_x^2 z^\nu_r)  -\nu h^2 z^\nu_r\right]\dd r \right \|^2_{L^2}\\
	&\quad\lesssim (t-s)\int_{s}^{t}[2\mathbb{E}[\|z^\nu_r\times \partial_x^2 z^\nu_r\|^2_{L^2}]+\mathbb{E}[\|h^2 z^\nu_r\|^2_{L^2}]]\dd r\\
	&\quad\lesssim (t-s)^2[2\|\partial_x h\|^2_{L^2}+\|h^2\|^2_{L^2}]\, .
\end{align*}
The Kolmogorov continuity theorem (see e.g. \cite[Theorem 2.3.11]{martina_book}) implies that the drift of the equation associated to $(z^\nu)_\nu$ is uniformly bounded in $\mathcal{L}^2(\Omega;C^{1/2-}([0,T];L^2))$. We estimate the stochastic integral. From Burkholder-Davis-Gundy's inequality and from Jensen's inequality, for $p>2$ and $\nu\in (0,1]$
\begin{align*}
\mathbb{E}\left\|\sqrt{\nu}\int_s^t h z^\nu_r\times \dd W_r\right\|^p_{L^2}
&\lesssim\mathbb{E}\left[\int_{s}^{t} \|hz^\nu_r\|^2_{L^2}\dd r\right]^{\frac{p}{2}} \\
&\lesssim(t-s)^{\frac{p}{2}-1}\int_{s}^{t}\mathbb{E}\left[ \|hz^\nu_r\|^{p}_{L^2}\right]\dd r\leq\|h\|^{p}_{L^2} (t-s)^{\frac{p}{2}}\, .
\end{align*}
	 From the Kolmogorov continuity theorem, there exists a modification of each stochastic integral in $z^\nu$ such that has $\mathbb{P}$-a.s.~H\"older continuous trajectories with exponent $\alpha\in (0,1/2)$ and such that, uniformly in $\nu$, the bound holds
	\begin{align*}
	\mathbb{E}\left[\left\|\sqrt{\nu}\int_0^\cdot h z^\nu_r\times\dd W_r\right\|_{C^{\alpha}(L^2)}^2\right]\lesssim  \|h\|^2_{L^2}\, .
	\end{align*}
	In conclusion, $(z^\nu)_\nu$ is uniformly bounded in $\mathcal{L}^2(\Omega;C^{\alpha}([0,T];L^2))$, for $\alpha \in (0,1/2)$.
\end{proof}

\subsection{Some Lemmata on the space average of the solution.}
For all $f\in L^1(D;\mathbb{R})$, we denote the space average by
	\begin{align*}
		\langle f\rangle :=\frac{1}{|D|}\int_D f(x)\dd x\, .
	\end{align*}
	We introduce a fundamental inequality to show the genuine randomness of the statistically stationary solutions.
\begin{lemma}\label{lemma:stochasticity_partial_x_h_0}
	Let $h\neq 0,\,\partial_x h=0$. For all fixed $\nu\in(0,1]$, let $z^\nu$ be a stationary solution to \eqref{LLG_nu}. Then for all $t\geq 0$, the equality holds
	\begin{align*}
	h^2 \mathbb{E}[z^\nu_t]=h^2 \mathbb{E}[\langle z^\nu_t\rangle]=0\, .
	\end{align*}
\end{lemma}
\begin{proof} What follows holds for all $t\geq 0$. We consider the stochastic LLG equation in the It\^o's formulation (see Lemma \ref{lemma:ito_stratonovich_correction}).
	If $\partial_x h=0$, then $\partial_x z^\nu_t=0$ $\mathbb{P}$-a.s. from Lemma \ref{lemma:a_priori_bound_derviative_p}. Thus for all $t\geq 0$ and for a.e. $x\in D$:
	\begin{align*}
	 z^\nu_t-z^\nu_s= \sqrt{\nu}\int_{s}^{t} h z^\nu_r\times \dd W_r-\nu \int_{s}^{t} h^2 z^\nu_r \dd r\, .
	\end{align*}
	
	We take the expectation and, from the stationarity of the solutions, we obtain the equality
	\begin{align*}
	\nu \mathbb{E}[h^2 z^\nu_t]=\nu h^2\mathbb{E}[ z^\nu_t]=0\, .
	\end{align*}
\end{proof}
 We need to show the following identity to prove that the solutions are non-trivial in space an time.
\begin{lemma}\label{lemma:condition_for_non_trivial} Let $h\in W^{1,\infty}(D;\mathbb{R})$ and $z^\nu$ be a stationary solution to \eqref{LLG_nu}. For every $t\geq 0$ and for all $\nu\in(0,1]$, the equality holds
	\begin{align*}
	\mathbb{E}[\langle z^\nu_t|\partial_x z^\nu_t|^2\rangle \cdot \langle z^\nu_t\rangle]-\mathbb{E}[\langle h^2 z^\nu_t \rangle\cdot \langle z^\nu_t\rangle]+\mathbb{E}[|\langle  h z^\nu_t\rangle |^2]=0\, .
	\end{align*}
	Let $u^\nu$ be a stationary solution to \eqref{mod_LLG_nu}. For every $t\geq 0$ and for all $\nu\in(0,1]$, the equality holds
	\begin{align*}
	\mathbb{E}[\langle u^\nu_t|\partial_x u^\nu_t|^2\rangle \cdot \langle u^\nu_t\rangle]-\mathbb{E}[\langle h^2 u^\nu_t \rangle\cdot \langle u^\nu_t\rangle]+\mathbb{E}[|\langle  h u^\nu_t\rangle |^2]=0\, .
	\end{align*}
\end{lemma}
\begin{proof}
	Consider the equation for the space average of the solution of the associated It\^o's formulation of the equation
	\begin{align*}
	\langle z^\nu_t\rangle-\langle z^\nu_s\rangle=\int_{s}^{t} A(z^\nu_r)\dd r+ \int_{s}^{t} \langle \sqrt{\nu} h z^\nu_r\rangle\times \dd W_r \, ,
	\end{align*}
	where $A(z^\nu_r)=\langle z^\nu_r\times \partial_x^2 z^\nu_r\rangle -\nu  \langle z^\nu_r\times [z^\nu_r\times \partial_x^2 z^\nu_r]\rangle -\nu \langle h^2z^\nu_r \rangle $. In the previous equality, we applied the stochastic Fubini's theorem to the stochastic integral (see e.g. \cite[Theorem IV.65]{protter}).
	We apply It\^o's formula to $F(x)=|x|^2$: we obtain that $F'(z^\nu)=2 z^\nu$ and $F''(z^\nu)=2I\in \mathbb{R}^{3\otimes 3}$. From It\^o's formula, 
	\begin{align*}
		 |\langle z^\nu_t\rangle|^2- |\langle z^\nu_s\rangle|^2=2\int_{s}^{t}
		A(z^\nu_r)\cdot \langle z^\nu_r\rangle \dd r&+\frac{1}{2}\int_{s}^{t}\mathrm{tr}\left(\Gamma(\langle \sqrt{\nu} h z^\nu_r\rangle )^T 2 I \Gamma(\langle \sqrt{\nu} h z^\nu_r\rangle )\right)\dd r\\
		&+2\int_{s}^{t}   \langle z^\nu_r\rangle\cdot\langle \sqrt{\nu} h z^\nu_r\rangle  \times \dd W_r\, ,
	\end{align*}
	where $\Gamma$ is defined in \eqref{eq:def_B}.
	As a consequence of Lemma \ref{lemma:traccia}, 
	\begin{align*}
		\mathrm{tr}\left(\Gamma(\langle \sqrt{\nu} h z^\nu_r\rangle )^T 2I \Gamma(\langle \sqrt{\nu} h z^\nu_r\rangle )\right)=4|\langle \sqrt{\nu} h z^\nu_r\rangle|^2\, .
	\end{align*}
	The drift has the form
	\begin{align*}
		A(z^\nu_r)\cdot \langle z^\nu_r\rangle&=[\langle z^\nu_r\times \partial_x^2 z^\nu_r\rangle -\nu  \langle z^\nu_r\times [z^\nu_r\times \partial_x^2 z^\nu_r]\rangle -\nu \langle h^2z^\nu_r \rangle ]\cdot \langle z^\nu_r\rangle \\
		&=[\langle z^\nu_r\times \partial_x^2 z^\nu_r\rangle +\nu  \langle \partial_x^2 z^\nu_r+z^\nu_r|\partial_x z^\nu_r|^2\rangle -\nu \langle h^2z^\nu_r \rangle ]\cdot \langle z^\nu_r\rangle \, ,
	\end{align*}
	where in the last equality we used \eqref{eq:u_cross_u_cross_Laplace}.
	Observe that, from the null Neumann boundary conditions on the gradient (which imply in one dimension that $\partial_x z^\nu=0$ on the boundary in the sense of the trace)
	\begin{align*}
	\langle z^\nu_r\times \partial_x^2 z^\nu_r\rangle  =\langle \partial_x(z^\nu_r\times \partial_x z^\nu_r) \rangle=0\, , \quad \quad \langle \partial_x^2 z^\nu_r\cdot \langle z^\nu_r\rangle \rangle=0\, .
	\end{align*}
	In conclusion the drift can be rephrased as
	\begin{align*}
	A(z^\nu_r)\cdot \langle z^\nu_r\rangle&=\nu  \langle z^\nu_r|\partial_x z^\nu_r|^2\rangle\cdot \langle z^\nu_r\rangle-\nu \langle h^2z^\nu_r \rangle \cdot \langle z^\nu_r\rangle \, .
\end{align*}
	This leads to the following expression 
	\begin{align*}
	|\langle z^\nu_t\rangle|^2 -|\langle z^\nu_s\rangle|^2&=2\nu\int_{s}^{t}\langle z^\nu_r|\partial_x z^\nu_r|^2\rangle \cdot \langle z^\nu_r\rangle \dd r-2\nu \int_{s}^{t}\langle h^2 z^\nu_r \rangle\cdot \langle z^\nu_r\rangle \dd r+2\nu \int_{s}^{t}|\langle  h z^\nu_r\rangle |^2\dd r\\
	&+2\int_{s}^{t}   \langle z^\nu_r\rangle\cdot\langle \sqrt{\nu} h z^\nu_r\times \dd W_r\rangle \, .
	\end{align*}
	By applying expectation, using the stationarity of the solutions and the fact that the stochastic integral is a martingale, the assertion follows.
	
	We proceed in the same way with $u^\nu$ and we obtain the equality
	\begin{align*}
		|\langle u^\nu_t\rangle|^2& -|\langle u^\nu_s\rangle|^2=2\nu\int_{s}^{t}\langle u^\nu_r|\partial_x u^\nu_r|^2\rangle \cdot \langle u^\nu_r\rangle \dd r- 2\int_{s}^{t}\langle (\sqrt{\nu}h+1)^2 u^\nu_r \rangle\cdot \langle u^\nu_r\rangle \dd r\\
		&+2 \int_{s}^{t}|\langle  (\sqrt{\nu}h+1) u^\nu_r\rangle |^2\dd r
		+2\int_{s}^{t}   \langle u^\nu_r\rangle\cdot\langle \sqrt{\nu} h u^\nu_r\times \dd W_r\rangle  +2\int_{s}^{t}   \langle u^\nu_r\rangle\cdot\langle  u^\nu_r \rangle\times  \dd W_r\, .
	\end{align*}
	Observe that the last stochastic integral vanishes by orthogonality. It also holds that
	\begin{align*}
		&\quad- \int_{s}^{t}\langle (\sqrt{\nu}h+1)^2 u^\nu_r \rangle\cdot \langle u^\nu_r\rangle \dd r+ \int_{s}^{t}|\langle  (\sqrt{\nu}h+1) u^\nu_r\rangle |^2\dd r\\
		&=- \int_{s}^{t}\langle (\nu h^2+1+2\sqrt{\nu}h) u^\nu_r \rangle\cdot \langle u^\nu_r\rangle \dd r+ \int_{s}^{t}\left[|\langle  \sqrt{\nu}h u^\nu_r\rangle|^2+2\langle  \sqrt{\nu}h u^\nu_r\rangle\cdot \langle   u^\nu_r\rangle +|\langle   u^\nu_r\rangle |^2\right]\dd r\\
		&= - \int_{s}^{t}\langle \nu h^2 u^\nu_r \rangle\cdot \langle u^\nu_r\rangle \dd r + \int_{s}^{t}|\langle  \sqrt{\nu}h u^\nu_r\rangle|^2\dd r\, .
	\end{align*}
	We can now apply expectation, use the martingale property of the stochastic integral and the stationarity of the solutions.
\end{proof}
The equality in Lemma \ref{lemma:condition_for_non_trivial} allows to show a further property of each stationary solution $z^\nu$.
\begin{corollary}\label{cor:inv_measure_LLG}
Let $h\in W^{1,\infty}(D;\mathbb{R})$ such that $\partial_x h \neq 0$ and $z^\nu$ be a stationary solution to \eqref{LLG_nu}. For every fixed $\nu\in(0,1]$, 
\begin{align*}
\mathbb{P}(\{\omega \in \Omega: \|\partial_x z^\nu_t\|^2_{L^2}>0\quad \forall t\geq 0\})=1\, .
\end{align*}
\end{corollary}
\begin{proof}
With a slight abuse of notation, we assume that the stationary solutions $(z^\nu)_{\nu}$ are defined on the whole real line: this can be achieved via a small modification of the Krylov-Bogolioubov argument.
We fix $\nu=1$ without loss of generality and we consider $z\equiv z^1$. Given a stationary solution $y$ to \eqref{LLG_nu}, it holds from Proposition \ref{pro:a_priori_approx_nu} and from the Poincaré inequality that $\mathbb{P}(\|\partial_x  y_t\|^2_{L^2}=0)=1  \iff \partial_x h=0$. 
This translates to $\mathbb{P}(\|\partial_x y_t\|^2_{L^2}=0)<1  \iff \partial_x h\neq 0$. In particular, there exists a set of positive measure $\mathbb{P}(\|\partial_x z_t\|^2_{L^2}>0)>0$. We prove that $\mathbb{P}(\|\partial_x z_t\|^2_{L^2}>0)=1$, provided $\partial_x h\neq0$. 
Denote by $\Gamma:=\{\omega\in \Omega: \|\partial_x z_t(\omega)\|_{L^2}=0 \quad \forall t\geq 0\}$, namely the set of events supporting solutions with trivial dynamic in space and time. By following the computations in Lemma \ref{lemma:condition_for_non_trivial} to the random variable $1_{\Gamma}z_t$, it follows 
\begin{align*}
\mathbb{E}[1_{\Gamma}\langle h^2 z_t \rangle\cdot \langle z_t\rangle]-\mathbb{E}[1_{\Gamma}|\langle  h z_t\rangle |^2]=\mathbb{E}\left[1_{\Gamma} \int_{0}^{t}   \langle z_r\rangle\cdot\langle  h z_r\times \dd W_r\rangle \right]\, ,
\end{align*}
for all $t\geq 0$. Now, using the fact that $z_t(\omega)=\langle z_t(\omega)\rangle$ for a.e. $x\in D$ and for all $\omega\in \Gamma$, we have
\begin{align*}
\mathbb{E}\left[1_{\Gamma} \int_{0}^{t}   \langle z_r\rangle\cdot\langle  h z_r\times \dd W_r\rangle \right]=\mathbb{E}\left[1_{\Gamma} \int_{0}^{t}   \langle z_r\rangle\cdot\langle z_r\rangle\langle h \rangle \times \dd W_r\right]=0\, ,
\end{align*}
from the elementary equality $a\,\cdot\, a\times b=0$ for all $a,b\in \mathbb{R}^3$.
Finally, by applying the fact that $z_t(\omega)=\langle z_t(\omega)\rangle$ for a.e. $x\in D$ and for all $\omega\in \Gamma$, we get the equation
\begin{align*}
(\langle h^2\rangle -\langle h \rangle^2) \mathbb{P}(\Gamma)=\mathbb{E}[1_{\Gamma}\langle h^2 z_t \rangle\cdot \langle z_t\rangle]-\mathbb{E}[1_{\Gamma}|\langle  h z_t\rangle |^2]=0\, .
\end{align*}
 Since $\partial_x h\neq 0$, we have $\langle h^2\rangle -\langle h \rangle^2\neq0$, which leads to the conclusion $\mathbb{P}(\Gamma)=0$.
\end{proof}

	\section{Proof of Theorem \ref{th:existence_statistical_solutions}}\label{sec:proof_theorem_1}

	\subsection{Proof of Theorem \ref{th:existence_statistical_solutions} a., b.}
	Consider the spaces
	\begin{align*}
		\mathcal{Y}^1_T:=L^2(0,T;H^1) \, ,
	\end{align*}
	\begin{align*}
		\mathcal{X}^1_T:=L^2 (0,T;H^2)\cap W^{\alpha,2}(0,T; L^2)\, ,
	\end{align*}
	for $\alpha\in (1/3,1/2)$. We prove that the family of invariant measures to the stochastic LLG equation \eqref{LLG_nu}, denoted by $(\mu^\nu)_{\nu}$, is tight on $\mathcal{Y}^1_T$. From Lemma \ref{lemma:tornstein}, the 
	ball $B_{1,R}:=\{ f\in \mathcal{X}^1_T: \|f\|_{\mathcal{X}^1_T}\leq R\}$ is relatively compact in $\mathcal{Y}^1_T$. From Chebyshev's inequality
	\begin{align*}
		\mu^\nu(B_{1,R}^C)\leq \frac{\int_{0}^{T}\mathbb{E}[\|z^\nu_r\|^2_{H^2}]\dd r+\mathbb{E}[\|z^\nu\|_{W^{\alpha,2}(0,T;L^2)}]}{R}\, .
	\end{align*}
	From Lemma \ref{cor:laplaciano_nu} and Lemma \ref{lemma:time_regularity}, the expectations in the numerator are uniformly bounded in $\nu$. 
	Consider the spaces  
	\begin{align*}
	\mathcal{Y}^2_T:= C_\mathrm{w}([0,T]; H^1)\, ,
	\end{align*}
	\begin{align*}
	\mathcal{X}^2_T:= L^\infty (0,T;H^1)\cap C^{\alpha}([0,T]; L^2)\, ,
	\end{align*}
	for $\alpha\in (1/3,1/2)$.
	From Lemma \ref{th:embedding_c_w_appendix}, the ball $B_{2,R}:=\{ f\in \mathcal{X}^2_T: \|f\|_{\mathcal{X}^2_T}\leq R\}$ is relatively compact in $\mathcal{Y}^2_T$. Again from Chebyshev's inequality and the uniform bounds in \eqref{eq:condition_stationarity_nu} in Lemma \ref{lemma:a_priori_bound_derviative_p} and Lemma \ref{lemma:time_regularity} of $(z^\nu)_\nu$ in $\mathcal{X}^2_T$,
		\begin{align*}
		\mu^\nu(B_{2,R}^C)\leq \frac{\mathbb{E}\left[\|z^\nu\|^2_{L^\infty (0,T;H^1)}\right]+\mathbb{E}[\|z^\nu\|_{C^{\alpha}(0,T;L^2)}]}{R}\, ,
		\end{align*}
	 and the sequence $(\mu^\nu)_\nu$ is tight on $\mathcal{Y}^2_T$.

	In conclusion, 
	$(\mu^\nu)_\nu$ is tight on $\mathcal{Y}_T:=\mathcal{Y}^1_T\cap \mathcal{Y}^2_T$ for all $T>0$. The tightness of $(\mu^\nu)_\nu$ on $\mathcal{Y}_T$ for every $T>0$ coincides with the tightness of $(\mu^\nu)_\nu$ on the locally convex topological space
	\begin{align}\label{eq:Y_grande}
		\mathcal{Y}:= L^2_{\mathrm{loc}}([0,+\infty);H^1)  \cap  C_{\mathrm{w}}([0,+\infty); H^1)\, .
	\end{align}
	where each space is endowed with the following topology:
	\begin{itemize}
		\item $L^2_{\mathrm{loc}}([0,+\infty);H^1)$ is endowed with $d_1(f,g):=\sum_{k\in \mathbb{N}}\frac{1}{2^k}\frac{\|f-g\|_{L^2(0,k;H^1)}}{1+\|f-g\|_{L^2(0,k;H^1)}}$,
		\item  $C_{\mathrm{w}}([0,+\infty); H^1)$ is endowed with semi-norms $\|f\|_{k,g}:=\sup_{t\in [0,k]}|\langle f_t,g\rangle|$, for $k\in\mathbb{N}$ and $g\in H^{-1}$.
	\end{itemize}
	The topology on $\mathcal{Y}$ is generated by the supremum of the topologies on each space.\\
	
	From the Skorokhod-Jakubowski Theorem (see e.g. \cite[Theorem 2.7.1]{martina_book}), there exists a probability space $(\hat{\Omega},\hat{\mathcal{F}},\hat{\mathbb{P}})$, a sequence $(\nu_j)_j$ converging to $0$ and random variables $(Z^{j})_j$, $Z$ on $(\hat{\Omega},\hat{\mathcal{F}},\hat{\mathbb{P}})$ and with values in $\mathcal{Y}$ such that 
	\begin{itemize}
		\item[1.] for every $j\in \mathbb{N}$, $Z^{j}$ and $z^{\nu_j}$ have the same law,
		\item[2.] $Z^{j}$ converges $\hat{\mathbb{P}}$-a.s. to $Z$ in $\mathcal{Y}$ as $j\rightarrow +\infty$.
	\end{itemize}
	Hence $(Z^j)_j$ converges in $C_\mathrm{w}([0,\infty); H^1)$ to $Z$. From \eqref{eq:condition_stationarity_nu}, 
	\begin{align*}
	\hat{\mathbb{E}}\left[\sup_{t\in [0,T]}\|\partial_x Z^j_t\|^2_{L^2}\right]=\mathbb{E}\left[\sup_{t\in [0,T]}\|\partial_x z^{\nu_j}_t\|^2_{L^2}\right]\lesssim\|\partial_x h\|^2_{L^2}\, .
	\end{align*}
	Thus $Z$ is stationary on $H^1$ from Proposition \ref{pro:sufficient_condition_stationary}.\\
	
	The considerations that follow hold for a.e. $t\geq 0$ and we make use of the stationarity of $z^{\nu_j}$. Recall from Lemma \ref{lemma:a_priori_bound_derviative_p} the uniform bound 
	\begin{align*}
	\hat{\mathbb{E}}[\|\partial_x Z^{j}_t\|_{L^2}^p]=\mathbb{E}[\|\partial_x z^{\nu_j}_t\|_{L^2}^p]\lesssim \|\partial_x h\|^p_{L^2}\, ,
	\end{align*}
	which holds for all $\nu\in(0,1]$ and for all $p\geq 2$. From Fatou's Lemma, the lower semi-continuity of the norm and the strong convergence in $L^2_{\mathrm{loc}}([0,\infty);H^1)$, the bound holds also for the limit $Z$
	\begin{align}\label{eq:inequality_partia_x_z}
	\hat{\mathbb{E}}[\|\partial_x Z_t\|_{L^2}^p]\lesssim \|\partial_x h\|^p_{L^2}\, .
	\end{align}
	  Since the sequences $(\partial_x^2 Z^{j})_j$, $(Z^{j}\times\partial_x^2 Z^{j})_j$ are uniformly bounded in $\mathcal{L}^2(\hat{\Omega}; L^2_{\mathrm{loc}}([0,\infty);L^2))$, they admit a weakly convergent subsequence (we identify from now on the sequence $(Z^j)_j$  with such a convergent subsequence). 
	We show that the weak limit of $(Z^{j}\times\partial_x^2 Z^{j})_j$ coincides with $Z\times\partial_x^2 Z$. 
	For a.e. $t>0$ and $\hat{\mathbb{P}}$-a.s., it holds that for all $\phi\in L^2$
	\begin{equation}\label{eq:equazione_limite}
	\begin{aligned}
	\lim_{j\rightarrow +\infty}\int_{0}^{T}\langle Z^j_t\times \partial_x^2 Z^j_t - Z_t\times \partial_x^2 Z_t ,\phi \rangle\dd t =\lim_{j\rightarrow +\infty }&-\int_{0}^{T}\langle Z_t\times (\partial_x  Z^j_t  -\partial_x Z_t ),\partial_x \phi \rangle \dd t\\
	&-\int_{0}^{T} \langle (Z^j_t-Z_t)\times \partial_x Z^j_t ,\partial_x \phi \rangle \dd t\, ,
	\end{aligned}
	\end{equation}
	where we used that $\partial_x (Z_t\times \partial_x Z_t)=Z_t\times \partial^2_x Z_t$, the integration by parts and the null Neumann boundary conditions.
	We look at the two expectations above separately. For the first term, we notice that
	\begin{align*}
	\int_{0}^{T}\left|\langle Z_t\times (\partial_x  Z^j_t  -\partial_x  Z_t ),\partial_x\phi \rangle \right|\dd t&=\left|(\langle Z_t\times \partial_x( Z^j_t  - Z_t )),\partial_x \phi \rangle \right|\\
	&\leq \int_{0}^{T} \|\partial_x (Z^j_t-Z_t)\|_{L^2}\|\partial_x \phi\|_{L^2}\dd t\, .
	\end{align*}
	From the Cauchy-Schwarz inequality and the strong convergence of $(Z^j)_j$ to $Z$ in $L^2_{\mathrm{loc}}([0,\infty);H^1)$, the first integral on the right hand side of \eqref{eq:equazione_limite} converges to $0$. For the second integral on the right hand side of \eqref{eq:equazione_limite}, we observe that
	\begin{align*}
	\int_{0}^{T}\left|\langle (Z^j_t-Z_t)\times \partial_x Z^j_t ,\partial_x \phi \rangle \right|\dd t\leq\int_{0}^{T} \|Z^j_t-Z_t\|_{L^\infty}\|\partial_x Z^j\|_{L^2}\|\partial_x \phi\|_{L^2}\dd t\, .
	\end{align*}
	Since  $H^1$ is continuously embedded in $L^\infty$ in one dimension, the above inequality implies for $T\geq 0$ and $\hat{\mathbb{P}}$-a.s. that
	\begin{align*}
	\lim_{j\rightarrow +\infty} \int_{0}^{T}\langle (Z^j_t-Z_t)\times \partial_x Z^j_t ,\partial_x \phi \rangle \dd t=0\, .
	\end{align*}
	In conclusion for a.e. $t\geq 0$ and $\hat{\mathbb{P}}$-a.s. 
	\begin{align*}
	\lim_{j\rightarrow +\infty} \int_{0}^{T}\langle Z^j_t\times \partial_x^2 Z^j_t - Z_t\times \partial_x^2 Z_t ,\phi \rangle \dd t=0\, .
	\end{align*}
	From the above considerations and from the weak convergence of $(Z^j\times \partial_x^2 Z^j)_{j}$ to $F$ in $\mathcal{L}^2(\hat{\Omega};\allowbreak L^2_{\mathrm{loc}}([0,\infty);L^2))$, it follows that for a.e. $t\geq 0$ and $\hat{\mathbb{P}}$-a.s. 
	\begin{align*}
	\lim_{j\rightarrow +\infty}\int_{0}^{T}\langle Z_t\times \partial_x^2 Z_t - F_t,\phi \rangle\dd t &=\lim_{j\rightarrow +\infty} \int_{0}^{T}\langle \langle Z_t\times \partial_x^2 Z_t -Z^j_t\times \partial_x^2 Z^j_t,\phi \rangle\dd t\\ &\quad+\lim_{j\rightarrow +\infty} \int_{0}^{T}\langle Z^j_t\times \partial_x^2 Z^j_t - F_t ,\phi \rangle\dd t \, .
	\end{align*}
	Hence we can identify the limit $F=Z\times \partial_x^2 Z$ as an equality in $L^2$. 	
	For the Laplacian, since it is a linear operator, the identification can be done as in the previous steps. From Fatou's Lemma, the lower semi-continuity of the norm on $\mathcal{L}^2(\hat{\Omega}; L^2_{\mathrm{loc}}([0,\infty);L^2))$ and the weak convergence, it holds that
	\begin{align*}
		\hat{\mathbb{E}} [\|\partial_x^2 Z_t\|^2_{L^2}]\leq C\, , \quad \hat{\mathbb{E}} [\|Z_t\times \partial_x^2 Z_t\|^2_{L^2}]\leq \|\partial_x h\|^2_{L^2}\, .
	\end{align*}
	From the equality \eqref{eq:fundamental_equality}, it follows that
	\begin{align*}
	\hat{\mathbb{E}} [\|\partial_x Z_t\|^4_{L^4}]\leq C\, .
	\end{align*}
	For all $j>0$, and for all $t\geq 0$ and $\hat{\mathbb{P}}$-a.s., it holds that $Z^{j}_t\in \mathbb{S}^2$,  hence also the limit $Z_t\in \mathbb{S}^2$ $\hat{\mathbb{P}}$-a.s. In particular $\hat{\mathbb{P}}$-a.s.~$Z\in L^\infty_{\mathrm{loc}}([0,+\infty);H^1)\cap L^2_{\mathrm{loc}}([0,+\infty);H^2)\cap C([0,+\infty);H^1)$. The previous bounds on the expectation prove Theorem \ref{th:existence_statistical_solutions} b.\\

	We now identify the limit equation satisfied by $Z$.  We recall that for every $T\geq 0$ and $\hat{\mathbb{P}}$-a.s., $Z^j$ fulfils the equation as the equality
	\begin{align}\label{eq:equazione_j_teorema}
		Z^j_T-Z^j_0-\int_{0}^{T}Z^j_t\times \partial_x^2 Z^j_t \dd t= -\nu_j\int_{0}^{T}Z^j_t\times(Z^j_t\times \partial_x^2 Z^j_t )\dd t + \Phi^{\nu_j}_{0,T}\, 
	\end{align}
	in $L^2$, where  $\Phi^{\nu_j}$ is defined implicitly by the equation and has the same law as
	\begin{align}\label{eq:stoch_int_z_nu_j}
		\sqrt{\nu_j}\int_{0}^{T} h z^{\nu_j}_t\times \circ \dd W_t\, .
	\end{align}
	The $\hat{\mathbb{P}}$-a.s.~limit for $j\rightarrow +\infty$ in $L^2$ of the left hand side of \eqref{eq:equazione_j_teorema} converges to 
	\begin{align*}
		Z_T-Z_0-\int_{0}^{T}Z_t\times \partial_x^2 Z_t \dd t
	\end{align*}
	from the $\hat{\mathbb{P}}$-a.s.~convergence of $(Z^j)_j$ to $Z$ in $C([0,+\infty);L^2)$ and the previous identification of the limit in the non-linearity. 
	The right hand side of \eqref{eq:equazione_j_teorema}  converges 
	 for $j\rightarrow +\infty$ to $0$ in $\mathcal{L}^1(\hat{\Omega};L^2)$: indeed, since $Z^j_t\in\mathbb{S}^2$, it holds
	\begin{align*}
	\nu^2_j\int_{0}^{T}\hat{\mathbb{E}}\|Z^j_t\times (Z^j_t\times \partial^2_x Z^j_t)\|^2_{L^2}\dd t\leq \nu^2_j\int_{0}^{T}\hat{\mathbb{E}}\|Z^j_t\times \partial^2_x Z^j_t\|^2_{L^2}\dd t \leq T\nu^2_j\|\partial_x h\|^2_{L^2}\, .
	\end{align*}
	To show that the term $\Phi^{\nu_j}_{0,T}$ in \eqref{eq:equazione_j_teorema} converges to $0$ for $j\rightarrow +\infty$, we recall that the law of $\Phi^{\nu_j}_{0,T}$ coincides with the law of \eqref{eq:stoch_int_z_nu_j}. Thus we can establish a vanishing bound in $\nu_j$ for the stochastic integral in \eqref{eq:stoch_int_z_nu_j}. From the It\^o-Stratonovich correction (see Lemma \ref{lemma:ito_stratonovich_correction}) 
	\begin{align}\label{eq:equazione_rumore}
	\sqrt{\nu_j} \int_{0}^{T} h z^{\nu_j}_t\times \circ \dd W_t=\sqrt{\nu_j} \int_{0}^{T} h z^{\nu_j}_t\times  \dd W_t-\nu_j\int_0^T h^2 z^{\nu_j}_t \dd t\, .
	\end{align}
	We apply the $L^2$-norm in space and take expectation in \eqref{eq:equazione_rumore}. We obtain
	\begin{align*}
		\mathbb{E}\left\|\sqrt{\nu_j} \int_{0}^{T} h z^{\nu_j}_t\times \circ \dd W_t\right\|_{L^2}\leq \mathbb{E}\left\|\sqrt{\nu_j} \int_{0}^{T} h z^{\nu_j}_t\times  \dd W_t\right\|_{L^2}+ \mathbb{E}\left\|\nu_j\int_0^T h^2 z^{\nu_j}_t \dd t\right\|_{L^2}\, .
	\end{align*}
	From the Burkholder-Davis-Gundy inequality and the properties of the Lebesgue integral, 
	\begin{align}\label{eq:bound_stoch_int_nu}
		\mathbb{E}\left\|\sqrt{\nu_j} \int_{0}^{T} h z^{\nu_j}_t\times \circ \dd W_t\right\|_{L^2}\leq (\sqrt{\nu_j}+\nu_j)\left(\int_0^T \mathbb{E}\left\|(h^2+h) z^{\nu_j}_t \right\|^2_{L^2} \dd t\right)^{1/2} \nonumber\\
		\leq (\sqrt{\nu_j}+\nu_j) T^{1/2}\| h^2 +h\|_{L^2}\, .
	\end{align}
	Thus the bound in \eqref{eq:bound_stoch_int_nu} holds also for $\hat{\mathbb{E}}[\|\Phi^{\nu_j}_{0,T}\|_{L^2}]$. This implies that $Z$ is a solution to the deterministic (SME) and the proof of Theorem \ref{th:existence_statistical_solutions} a.~is concluded.
	
	We show that the limit $Z$ has null Neumann boundary conditions. For all $\phi\in H^1$,
	\begin{align*}
		\int_D Z^{\nu_j}_t\times \partial_x^2 Z^{\nu_j}_t\cdot \phi \dd x=-\int_D Z^{\nu_j}_t\times \partial_x Z^{\nu_j}_t\cdot \partial_x\phi \dd x+\int_D\partial_x[ Z^{\nu_j}_t\times \partial_x Z^{\nu_j}_t\cdot \phi ]\dd x\, .
	\end{align*}
	For all $j\in \mathbb{N}$ it holds from the divergence theorem and the null Neumann boundary conditions on $Z^{\nu_j}$,
	\begin{align*}
	\int_D\partial_x[ Z^{\nu_j}_t\times \partial_x Z^{\nu_j}_t\cdot \phi ]\dd x=0\, .
	\end{align*}
	By passing to the limit $j\rightarrow +\infty$, the null Neumann boundary conditions pass to the limit solution $Z$.
	We prove now Theorem \ref{th:existence_statistical_solutions} c.: we discuss the existence of conservation laws for the limit solution $Z$. In what follows, we work for all $T\geq 0$ and $\hat{\mathbb{P}}$-a.s. From the previous step, $Z$ is a solution to the deterministic (SME) and by considering its spatial average 
	\begin{align*}
		\langle Z_T\rangle = \langle Z_0\rangle+\int_{0}^{T}\langle Z_t\times \partial^2_x Z_t\rangle \dd t=\langle Z_0\rangle \, ,
	\end{align*}
	since $ \langle Z_t\times \partial^2_x Z_t\rangle=\langle \partial_x(Z_t\times \partial_x Z_t)\rangle=0$ for a.e. $t> 0$ and $\hat{\mathbb{P}}$-a.s.~from the null Neumann boundary conditions.
	We look now at the time evolution of $\|\partial_x Z_T\|^2_{L^2}$, 
	\begin{align*}
		\|\partial_x Z_T\|^2_{L^2}=\|\partial_x Z_0\|^2_{L^2}+\int_{0}^{T}\langle \partial_x [Z_t\times \partial_x^2 Z_t], \partial_x Z_t\rangle \dd t\, . 
	\end{align*}
	The Lebesgue integral on the right hand side vanishes from the integration by parts formula, the null Neumann boundary conditions and the elementary equality $(a\times b)\cdot b=0$ for all $a,b\in \mathbb{R}^3$. In conclusion, $t\mapsto \|\partial_x Z_t\|^2_{L^2}$ is a constant map.
	
	By looking at the other conservation law in the statement, the equality $\|Z_T\|_{L^2}=\|Z_0\|_{L^2}$ holds trivially from the spherical bound. Moreover, $\langle Z_t,Q\rangle_{L^2}=\langle Z_0,Q\rangle_{L^2}$ for all $t\geq 0$. Hence the map $t\mapsto \|Z_t-Q\|^2_{L^2}$ is constant $\hat{\mathbb{P}}$-a.s.

\section{Proof of Theorem \ref{th:intro_solution_stochastic}.}\label{sec:proof_theorem_4_stochasticity}
This section is devoted to the proof of Theorem \ref{th:intro_solution_stochastic}. We underline the core elements of each step.
\begin{itemize}
\item[a.] Theorem \ref{th:intro_solution_stochastic} a.: The easy implication is $\partial_x h=0$ implies $\partial_x Z=0$: we use Lemma \ref{lemma:a_priori_bound_derviative_p} on $\partial_x Z^j=0$ and the claim follows, by passing to the limit.
The other implication is less trivial and it relies on the equality in Lemma \ref{lemma:condition_for_non_trivial}: Lemma \ref{lemma:condition_for_non_trivial} is used at the level of the approximations, then the property is passed to the limit equation. The resulting equality allows to conclude that $h$ can be viewed as a random variable with null variance, thus it must be constant.
\item[b.] Theorem \ref{th:intro_solution_stochastic} b.: is a consequence of Theorem \ref{th:intro_solution_stochastic} a.~and other immediate considerations on the expectation of the statistically stationary solution.
\item[c.] Theorem \ref{th:intro_solution_stochastic} c.: is a consequence of Theorem \ref{th:solution_stochastic} below and of the Poincaré inequality.
\item[d.] Theorem \ref{th:intro_solution_stochastic} d., e.: are corollaries of the previous steps and of Lemma \ref{lemma:condition_for_non_trivial}.

\end{itemize}
\subsection{Proof of Theorem \ref{th:intro_solution_stochastic} a.}
We prove the following assertion: fix  $h\neq 0$, $\partial_x h=0$ if and only if $\partial_x Z=0$ for a.e.~$x\in D$ and $\hat{\mathbb{P}}$-a.s.
Then the claim follows. \\

We start from the easier case: assume $\partial_x h=0$. 	
In this case, by Lemma \ref{lemma:a_priori_bound_derviative_p}, $\partial_x Z^{j}=0$ for all $j\in \mathbb{N}$, for a.e.~$x\in D$ and $\hat{\mathbb{P}}$-a.s.: hence, from the strong convergence in $L^2_{\mathrm{loc}}([0,\infty);H^1)$, also $\partial_x Z=0$ for a.e.~$x\in D$ and $\hat{\mathbb{P}}$-a.s. \\

Assume now that $\partial_x Z=0$. From the Poincaré-Wirtinger inequality (see Lemma \ref{teo:Poin_wirt}), $Z=\langle Z\rangle$ $\hat{\mathbb{P}}$-a.s.~and for a.e.~$x\in D$.  In this case, for all $j\in \mathbb{N}$ and for all $t\geq 0$ recall the identity in Lemma \ref{lemma:condition_for_non_trivial} 
\begin{align}\label{eq:label_IIF_a}
\hat{\mathbb{E}}[\langle Z^j_t|\partial_x  Z^j_t|^2\rangle \cdot \langle Z^j_t\rangle]-\hat{\mathbb{E}}[\langle h^2 Z^j_t \rangle\cdot \langle Z^j_t\rangle]=-\hat{\mathbb{E}}[|\langle  h Z^j_t\rangle |^2]\, .
\end{align}
Since $Z^j_t\in \mathbb{S}^2$ $\hat{\mathbb{P}}$-a.s., for all $t\geq 0$ and for $x\in D$ and $h(x)\in \mathbb{R}$ for a.e.~$x\in D$, the sequences $(\langle h^2 Z^j_t \rangle\cdot \langle Z^j_t\rangle)_j$ and $(|\langle  h Z^j_t\rangle |^2)_j$ are uniformly integrable. From Proposition \ref{lemma:a_priori_bound_derviative_p} and the spherical bounds, the sequence $(\langle Z^j_t|\partial_x  Z^j_t|^2\rangle \cdot \langle Z^j_t\rangle)_j$ is uniformly integrable, indeed 
\begin{align*}
\hat{\mathbb{E}}[|\langle Z^j_t|\partial_x Z^j_t|^2\rangle \cdot \langle Z^j_t\rangle|^2]\lesssim_D \hat{\mathbb{E}}[\|\partial_x Z^j_t\|^4_{L^4}]\, .
\end{align*}
We can pass to the limit under the sign of expectation in equality \eqref{eq:label_IIF_a} and in the space integral. From the strong convergence in $L^2_{\mathrm{loc}}([0,\infty);H^1)$ and the hypothesis $\partial_x Z=0$,  
\begin{align}\label{eq:equality_h_constant}
\hat{\mathbb{E}}[\langle h^2 Z_t \rangle\cdot \langle Z_t\rangle]=\hat{\mathbb{E}}[|\langle  h Z_t\rangle |^2]\, .
\end{align}
Recall now that $Z_t=\langle Z_t  \rangle$ $\hat{\mathbb{P}}$-a.s., for a.e.~$x\in D$ and $t\geq 0$, thus $|Z_t|=|\langle Z_t  \rangle|=1$. We can rewrite the left hand side of \eqref{eq:equality_h_constant} as
\begin{align*}
\hat{\mathbb{E}}[\langle h^2 Z_t \rangle\cdot \langle Z_t\rangle]= \hat{\mathbb{E}}[\langle h^2\rangle |\langle  Z_t \rangle|^2]=\langle h^2\rangle \, .
\end{align*}
With analogous considerations, we rewrite the right hand side of \eqref{eq:equality_h_constant} as
\begin{align*}
\hat{\mathbb{E}}[|\langle  h Z_t\rangle |^2]=\hat{\mathbb{E}}[|\langle  h\rangle |^2| \langle Z_t\rangle |^2]= |\langle  h\rangle |^2\, .
\end{align*}
This implies that $\langle h^2\rangle =|\langle  h\rangle |^2$, which holds only if $x\mapsto h(x)$ is constant in space. Indeed, we view $x\mapsto h(x)$ as a random variable defined on the probability space $(D,\mathcal{B}_D, \dd x/|D|)$: the relation $\langle h^2\rangle =|\langle  h\rangle |^2$ says that $h$ is a random variable with null variance. Thus $\langle (h-\langle  h\rangle )^2 \rangle =0$, which implies that $h-\langle  h\rangle =0$ for a.e.~$x\in D$ and therefore $h$ is constant in space.

\subsection{Theorem \ref{th:intro_solution_stochastic} b. }
We show that the map $t\mapsto Z_t$ is constant for a.e. $x\in D$ and $\hat{\mathbb{P}}$-a.s. if and only if $\partial_x h=0$. We start by showing that if $t\mapsto Z_t$ is constant for a.e. $x\in D$ and $\hat{\mathbb{P}}$-a.s. then $\partial_x h=0$: the assumption implies that $Z_t\times \partial_x^2 Z_t=0$ for a.e.~$x\in D$, $\hat{\mathbb{P}}$-a.s.~and for all $t\geq0$. From the equality \eqref{eq:label_IIF_a} and the same considerations as in the proof of Theorem 1.2 a., it follows that $h=\langle h\rangle$ for a.e.~$x\in D$ and thus $\partial_x h=0$.
Assume now that $\partial_x h=0$, then from inequality \eqref{eq:inequality_partia_x_z} it holds $\partial_x Z_t=0$ $\hat{\mathbb{P}}$-a.s., for a.e.~$x\in D$ and for all $t\geq 0$. Thus the non-linearity in the SME vanishes and the map $t\mapsto Z_t$ is constant for a.e. $x\in D$ and $\hat{\mathbb{P}}$-a.s.
Note that in the above we have shown that
\begin{align*}
\hat{\mathbb{P}}\left(\{\hat{\omega}\in \hat{\Omega}:t\mapsto Z_t(\hat{\omega})\; \mathrm{is }\; \mathrm{constant\;  for\; a.e. } \; x\in D\}\right)= \hat{\mathbb{P}}\left(\{\hat{\omega}\in \hat{\Omega}:\|\partial_x  Z_t(\hat{\omega})\|_{L^2}=0 \quad \forall t\geq 0 \}\right)\, .
\end{align*}
The second part of Theorem \ref{th:intro_solution_stochastic} a.~together with the proved assertion implies that $\partial_x h\neq 0$ if and only if
\begin{align*}
\hat{\mathbb{P}}\left(\{\hat{\omega}\in \hat{\Omega}:t\mapsto Z_t(\hat{\omega})\; \mathrm{is }\; \mathrm{constant\;  for\; a.e. } \; x\in D\}\right)=0\, .
\end{align*}

\subsection{Theorem \ref{th:intro_solution_stochastic} c.}
We show that the solutions are genuinely stochastic.
\begin{theorem}\label{th:solution_stochastic}
	Assume that $\partial_x h=0$, then $\partial_x Z=0$, $Z=\langle Z\rangle$ $\hat{\mathbb{P}}$-a.s. and for all $t\geq 0$
	\begin{align*}
	\hat{\mathbb{E}}[Z_t]=0\, .
	\end{align*}
	Assume that $\partial_x h\neq 0$,  then for all $t\geq 0$ it holds  $\hat{\mathbb{P}}(\|\partial_x Z_t\|_{L^2}>0)>0$ and
	\begin{align*}
	\hat{\mathbb{E}}[Z_t\times \partial_x^2 Z_t]=0\, .
	\end{align*}
\end{theorem}
\begin{proof}
	What follows holds for all $t\geq 0$.	
	If $\partial_x h \neq 0$, then from Theorem \ref{th:intro_solution_stochastic} a.~$\hat{\mathbb{P}}(\|\partial_x Z_t\|^2_{L^2}>0)>0$. From the Poincaré inequality, $\|Z_t\times \partial_x^2 Z_t\|^2_{L^2}>0$ on a set of positive probability. We take expectation, 
	\begin{align*}
		\hat{\mathbb{E}}[Z_t]=\hat{\mathbb{E}}[Z_0]+\int_{0}^{t} \hat{\mathbb{E}}[Z_r\times \partial_x^2 Z_r ]\dd r
	\end{align*}
	and, by using the stationarity of the solution $Z$, 
	\begin{align*}
	\hat{\mathbb{E}}[Z_t\times \partial_x^2 Z_t]=0\, .
	\end{align*}
	Thus $Z_t\times \partial_x^2 Z_t\neq 0$ with positive probability, but its space average is $0$: this implies that the map $\hat{\omega} \mapsto Z_t(\hat{\omega})\times \partial_x^2 Z_t(\hat{\omega})$ is stochastic.\\
	
	If $\partial_x h=0$, it holds $Z_t=\langle Z_t \rangle$ $\hat{\mathbb{P}}$-a.s. and $|Z_t|=|\langle Z_t \rangle|=1$. From Lemma \ref{lemma:stochasticity_partial_x_h_0}, it holds that $\hat{\mathbb{E}}[\langle Z_t \rangle]=0$: hence $\hat{\omega}\mapsto Z_t(\hat{\omega})=\langle  Z_t(\hat{\omega}) \rangle $ is stochastic.
\end{proof}

\subsection{Theorem \ref{th:intro_solution_stochastic} d. }
\textcolor{black}{We prove that $Z$ is not a sum of a countable number of indicator functions if $\partial_x h\neq 0$.}
Assume that $Z(\hat{\omega})=\sum_{j=1}^{\infty}g^j1_{\hat{\Omega}_j}$ for all $\hat{\omega}\in \hat{\Omega}$, where $\{\hat{\Omega}_j\}_j$ is a partition of $\hat{\Omega}$ such that $\hat{\Omega}=\dot{\bigcup} \hat{\Omega}_j$ and $g^j$ solutions to the SME. 
We compute the expectation of the random variables $Z_t 1_{\hat{\Omega}_j}$ and we use the stationarity of the process, for all $t\geq 0$
\begin{align*}
	g^j_t \hat{\mathbb{P}}(\hat{\Omega}_j)=\hat{\mathbb{E}}[Z_t 1_{\hat{\Omega}_j}]=\hat{\mathbb{E}}[Z_0 1_{\hat{\Omega}_j}]=g^j_0\hat{\mathbb{P}}(\hat{\Omega}_j)\, ,\quad \forall j\in \{1,\cdots,n\}\, .
\end{align*}
Hence every $t\mapsto g^j_t$ is constant. On the other hand, from Theorem \ref{th:intro_solution_stochastic} b., the map $t\mapsto Z_t$ is not constant. Thus $Z$ cannot be decomposed in a countable sum of indicator functions.

\subsection{Theorem \ref{th:intro_solution_stochastic} e. }
Assume now that $Z_t(x,\hat{\omega})=\sum_{j=1}^n g^j_t(\hat{\omega}) 1_{A_j}(x)$ for $A_j\subset\mathcal{B}_D$ disjoint and constituting a partition of $D$ and for $x\mapsto g^j_t(x,\omega)$ constant $\hat{\mathbb{P}}$-a.s.~in $A_j$. On each $A_j$, from a small modification of Lemma \ref{lemma:condition_for_non_trivial}, 
\begin{align*}
\hat{\mathbb{E}}[\langle 1_{A_j} Z_t|\partial_x Z_t|^2\rangle \cdot \langle 1_{A_j} Z_t\rangle]-\hat{\mathbb{E}}[\langle 1_{A_j}  h^2 Z_t \rangle\cdot \langle 1_{A_j} Z_t\rangle]+\hat{\mathbb{E}}[|\langle 1_{A_j}  h Z_t\rangle |^2]=0\, .
\end{align*}
Since $Z_t1_{A_j}=g^j_t$ is constant in space, then $Z_t1_{A_j}=\langle Z_t1_{A_j}\rangle =\langle g^j_t \rangle$ and the previous equation reads
\begin{align*}
-\hat{\mathbb{E}}[\langle 1_{A_j}  h^2 Z_t \rangle\cdot \langle 1_{A_j} Z_t\rangle]+\hat{\mathbb{E}}[|\langle 1_{A_j}  h Z_t\rangle |^2]=-\hat{\mathbb{E}}[\langle  h^2  \rangle|g^j_t|^2]+\hat{\mathbb{E}}[|\langle  h \rangle |^2|g^j_t|^2]=0\, .
\end{align*}
This leads to $h=\langle h\rangle$: this cannot be, since $\partial_x h\neq 0$.

\section{An estimate from below for the set of time and space dynamic trajectories.}\label{sec:lower_bound}
Let $Z$ be a statistical solution as constructed in Theorem \ref{th:intro_solution_stochastic} from the stochastic LLG equation with $h\in W^{1,\infty}$, with $\partial_x h\neq 0$.
The stationary statistical solution $Z$ satisfies for all $t\geq 0$ the equality
\begin{align}\label{eq:label_IIF_limit}
\hat{\mathbb{E}}[\langle Z_t|\partial_x  Z_t|^2\rangle \cdot \langle Z_t\rangle]-\hat{\mathbb{E}}[\langle h^2 Z_t \rangle\cdot \langle Z_t\rangle]=-\hat{\mathbb{E}}[|\langle  h Z_t\rangle |^2]\, .
\end{align}
We partition $\hat{\Omega}=\hat{\Gamma}\; \dot{\cup}\; \hat{\Gamma}^C$, where
$\hat{\Gamma}:=\{\hat{\omega}\in \hat{\Omega}: \|\partial_x Z_t(\hat{\omega})\|^2_{L^2}=0  \quad\forall t\geq 0\}\, .$ We can then rewrite \eqref{eq:label_IIF_limit} as
\begin{align}\label{eq:proof_lower_bound}
2\hat{\mathbb{E}}[1_{\hat{\Gamma}^C}\langle Z_t|\partial_x  Z_t|^2\rangle \cdot \langle Z_t\rangle]-2\hat{\mathbb{E}}[1_{\hat{\Gamma}^C}\langle h^2 Z_t \rangle\cdot \langle Z_t\rangle]+2\hat{\mathbb{E}}[1_{\hat{\Gamma}^C}|\langle  h Z_t\rangle |^2]=2(\langle h^2\rangle -\langle h \rangle^2) \hat{\mathbb{P}}(\hat{\Gamma})\, .
\end{align}
Observe that we can rewrite $\langle h^2 Z_t \rangle\cdot \langle Z_t\rangle$ as
\begin{align*}
-2\langle h^2 Z_t \rangle\cdot \langle Z_t\rangle
&=\langle |h Z_t - h \langle Z_t\rangle|^2\rangle -\langle |h \langle Z_t\rangle|^2 \rangle-\langle |hZ_t|^2\rangle\, .
\end{align*}

Note that $\hat{\mathbb{E}}[|\langle  h Z_t\rangle |^2]-\hat{\mathbb{E}}[|
\langle  |h Z_t|^2\rangle ]\leq 0$, since we can see it as the expectation of minus the spatial variance of the random variable $ h Z_t$ on the probability space $(D,\mathcal{B}_D, \dd x/|D|)$.
In view of the previous considerations, we rewrite \eqref{eq:proof_lower_bound} as
\begin{align}\label{eq:main_ineq}
2\hat{\mathbb{E}}[1_{\hat{\Gamma}^C}\langle Z_t|\partial_x  Z_t|^2\rangle \cdot \langle Z_t\rangle]+\hat{\mathbb{E}}[1_{\hat{\Gamma}^C}\langle |h Z_t - h \langle Z_t\rangle|^2\rangle]+\hat{\mathbb{E}}[1_{\hat{\Gamma}^C}|\langle  h Z_t\rangle |^2]>2(\langle h^2\rangle -\langle h \rangle^2) \hat{\mathbb{P}}(\hat{\Gamma})\, .
\end{align}
From the Cauchy-Schwarz inequality,
\begin{align*}
2\hat{\mathbb{E}}[1_{\hat{\Gamma}^C}\langle Z_t|\partial_x  Z_t|^2\rangle \cdot \langle Z_t\rangle]\leq 2\frac{\hat{\mathbb{P}}(\hat{\Gamma}^C)^{1/2}}{|D|} \hat{\mathbb{E}}[\|\partial_x Z_t\|^4_{L^2}]^{1/2}\, .
\end{align*}
We follow the lines of  Lemma \ref{lemma:a_priori_bound_derviative_p}: by evaluating \eqref{eq:00} for $p=4$, it holds
\begin{align}
\mathbb{E}[\|\partial_x Z_t\|^{2}_{L^2}\|Z_t\times \partial^2_x Z_t\|_{L^2}^2 ]&=\mathbb{E}[\|\partial_x Z_t\|^{2}_{L^2}\|\partial_x h\|^2_{L^2}]+\mathbb{E}\left[\left|\int_{D} \partial_x h Z_t\times \partial_x Z_t\dd x\right|^2 \right] \nonumber\, .
\end{align}
Using this fact together with the Poincaré inequality and the Cauchy-Schwarz inequality, 
\begin{align*}
\hat{\mathbb{E}}[\|\partial_x Z_t\|^4_{L^2}]\leq C_p \hat{\mathbb{E}}[\|\partial_x Z_t\|^2_{L^2} \|Z_t\times\partial^2_x Z_t\|^2_{L^2}]\leq 2C_p \|\partial_x h\|^2_{L^2} \hat{\mathbb{E}}[ \|\partial_x Z_t\|^2_{L^2}]\leq 2C_p^2 \|\partial_x h\|^4_{L^2} \, ,
\end{align*}
from which the estimate
\begin{align}\label{eq:ineq_1}
2\hat{\mathbb{E}}[1_{\hat{\Gamma}^C}\langle Z_t|\partial_x  Z_t|^2\rangle \cdot \langle Z_t\rangle]\leq 2\frac{\hat{\mathbb{P}}(\hat{\Gamma}^C)^{1/2}}{|D|} [2C_p^2\|\partial_x h\|^4_{L^2}]^{1/2}\, .
\end{align}
Again from the Cauchy-Schwarz inequality and the fact that $|Z_t|_{\mathbb{R}^3}=1$, it holds
\begin{align}\label{eq:ineq_2}
\hat{\mathbb{E}}[1_{\hat{\Gamma}^C}|\langle  h Z_t\rangle |^2]\leq  \langle |h| \rangle^2\hat{\mathbb{P}}(\hat{\Gamma}^C)\, .
\end{align}
The Cauchy-Schwarz inequality, the Poincaré-Wirtinger inequality (see Theorem \ref{teo:Poin_wirt}) and the fact that $h\in W^{1,\infty}$ imply
\begin{equation}
\begin{aligned}\label{eq:ineq_3}
\hat{\mathbb{E}}[1_{\hat{\Gamma}^C}\langle |h Z_t - h \langle Z_t\rangle|^2\rangle]
&\leq C_p\|h\|^2_{L^\infty} \frac{\hat{\mathbb{P}}(\hat{\Gamma}^C)^{1/2}}{|D|} \hat{\mathbb{E}}[\|\partial_x Z_t\|^4_{L^2}]^{1/2}\\
&\leq C_p \|h\|^2_{L^\infty} \frac{\hat{\mathbb{P}}(\hat{\Gamma}^C)^{1/2}}{|D|} [2C_p^2\|\partial_x h\|^4_{L^2}]^{1/2}\, .
\end{aligned}
\end{equation}
By inserting \eqref{eq:ineq_1}, \eqref{eq:ineq_2} and \eqref{eq:ineq_3} into \eqref{eq:main_ineq}, and by substituting $\hat{\mathbb{P}}(\hat{\Gamma})=1-\hat{\mathbb{P}}(\hat{\Gamma}^C)$ we the estimate from below on $\hat{\mathbb{P}}(\hat{\Gamma^C})$
\begin{align}\label{eq:estimate_below}
(2+\|h\|^2_{L^\infty}C_p)\frac{\hat{\mathbb{P}}(\hat{\Gamma}^C)^{1/2}}{|D|} \sqrt{2}C_p\|\partial_x h\|^2_{L^2}+[ \langle|h| \rangle^2+2(\langle h^2\rangle -\langle h \rangle^2)]\hat{\mathbb{P}}(\hat{\Gamma}^C)>2(\langle h^2\rangle -\langle h \rangle^2) \, .
\end{align}
By fixing $h$, inequality \eqref{eq:estimate_below} can be used to compute a lower bound for the probability $\hat{\mathbb{P}}(\hat{\Gamma}^C)$. Observe that the lower bound obtained in inequality \eqref{eq:estimate_below} can be improved, for instance, by using different weights in the Cauchy-Schwarz inequality in \eqref{eq:ineq_1} and \eqref{eq:ineq_3}: giving a power closer $1$ to $\hat{\mathbb{P}}(\hat{\Gamma}^C)$ can optimize the estimate. We leave this fact to the interested reader.

As an explicit example, consider $h(x)=\alpha\cos(x)$ on $D:=[0,2\pi k]$, for $\alpha>0$ and $k\in \mathbb{N}$. In this case $\langle h\rangle=0$,  $\langle h^2\rangle =\langle [\partial_x h]^2\rangle=k\pi/(2 k \pi)=1/2 $, $C_p=1$, $\langle|h| \rangle^2=(4k)^2/(2 k \pi)^2=4/\pi^2$. By setting $\lambda=\hat{\mathbb{P}}(\hat{\Gamma}^C)$, inequality \eqref{eq:estimate_below} reads
\begin{align*}
\lambda\left(1+\frac{4}{\pi^2}\right)+\frac{\sqrt{2}(2+\alpha^2)}{2}\sqrt{\lambda}-1>0\, .
\end{align*}
By choosing $\alpha=0.1$, we conclude that $\hat{\mathbb{P}}(\hat{\Gamma}^C)>0.2298$.

\section{Existence of stationary solutions to the stochastic Schr\"odinger map equation.}\label{sec:stationary_solutions}
A small modification of the previous argument allows us to prove existence of stationary solutions to the stochastic SME, 
\begin{align}\label{SSME}
 u_{t}-u_s=\int_{s}^{t} u_r\times \partial_x^2 u_r \dd r+\int_{s}^{t} u_r\times \circ \dd W_r\, ,
\end{align}
with null Neumann boundary conditions.
We approximate the SSME by means of the LLG equation in 1D with spherical noise, i.e.~for every $\nu\in(0,1]$ we consider a stationary solution $z^\nu$  to the stochastic Landau-Lifschitz-Gilbert equation in 1D (see Proposition \ref{pro:a_priori_approx_nu}),
\begin{align}\label{LLG_2_nu}
 z^\nu_t-z^\nu_s=\int_{s}^{t} z^\nu_r\times \partial_x^2 z^\nu_r\dd r-\nu \int_{s}^{t} z^\nu_r\times [z^\nu_r\times \partial_x^2 z^\nu_r]\dd r+ \int_{s}^{t} [\sqrt{\nu} h+1] z^\nu_r\times \circ \dd W_r\, .
\end{align}
The a priori estimates are the same as above, thus it follows that there exists a stationary solution in this case. Before stating the main result of this section, we provide a definition of stationary martingale solution that we employ.
\begin{definition} (Stationary martingale solution)
We say that $(\hat{\Omega}, \hat{\mathcal{F}},(\hat{\mathcal{F}}_t)_t,\hat{\mathbb{P}}, B,Z )$ is a \textit{stationary martingale solution} to \eqref{SSME} if
\begin{itemize}
\item $(\hat{\Omega}, \hat{\mathcal{F}}, (\hat{\mathcal{F}}_t)_t,\hat{\mathbb{P}})$ is a complete filtered probability space,
\item $B$ is a $\mathbb{R}^3$-valued $(\hat{\mathcal{F}}_t)_t$-Brownian motion,
\item $Z$ is a $H^1$-valued stochastic process stationary on $H^1$, whose trajectories belong $\hat{\mathbb{P}}$-a.s.~to $C([0,+\infty);L^2)$ and such that for all $t>0$ the process $(Z,B)$ fulfils \eqref{SSME} $\hat{\mathbb{P}}$-a.s.~in $L^2$.
\end{itemize}
\end{definition}
In the remaining part of this section, we prove Theorem \ref{th:existence_stationary}.
\begin{proof} (of Theorem \ref{th:existence_stationary}) Note that $Y$ in the statement corresponds to $Z$ in the proof that follows.
	We approximate the stochastic Schr\"odinger map equation by means of the stochastic Landau-Lifschitz-Gilbert equation \eqref{LLG_2_nu}.  The energy equalities read
	\begin{align*}
	|z^\nu_t|_{\mathbb{R}^3}^2=|z^\nu_0|_{\mathbb{R}^3}^2\, ,
	\end{align*}
	\begin{align*}
	\mathbb{E}\|\partial_x z^\nu_t\|^2_{L^2}=\mathbb{E}\|\partial_x z^\nu_0\|^2_{L^2}-\nu\int_{0}^{t}\mathbb{E}\|z^\nu_r\times \partial_x^2 z^\nu_r\|^2_{L^2}\dd r+t\nu\|\partial_x h\|^2_{L^2}\, .
	\end{align*}
	The a priori estimates coincide and the tightness argument is very similar to the proof of Theorem \ref{th:existence_statistical_solutions}: the main difference consists in the fact that we need to identify a martingale solution and the stochastic integral on the new probability space. We prove in detail this last assertion. Define $\mu_W$ to be the law of the Brownian motion on $(\Omega,\mathcal{F},\mathbb{P})$. Define the space $\mathcal{Y}_W:=C([0,+\infty); \mathbb{R}^3)$ and let $\mathcal{Y}$ be as in \eqref{eq:Y_grande}. Then $\mu_W$ is tight on $\mathcal{Y}_W$. 
	From Tichonoff's Theorem, $\mu_\nu \otimes \mu_W$ is tight on $\mathcal{Y}\times \mathcal{Y}_W$.  From the Skorokhod-Jakubowski theorem, there exists a new probability space $(\hat{\Omega}, \hat{\mathcal{F}}, \hat{\mathbb{P}})$, a sequence $(\nu_j)_j$ converging to $0$ and random variables $(Z^j, B^j)$, $(Z,B)$ with values in  $\mathcal{Y}\times \mathcal{Y}_W$ such that
	\begin{itemize}
		\item $(Z^j, B^j)$ and $(z^{\nu_j}, W)$ have the same law
		\item $(Z^j, B^j)$  converges $\hat{\mathbb{P}}$-a.s. to $(Z, B)$ on $\mathcal{Y}\times \mathcal{Y}_W$.
	\end{itemize}
	Define $(\hat{\mathcal{F}}^j_t)_t$ to be the $\hat{\mathbb{P}}$-augmented canonical filtration generated by $(Z^j, B^j)$ and $(\hat{\mathcal{F}}_t)_t$ to be the $\hat{\mathbb{P}}$-augmented canonical filtration generated by $(Z, B)$.
	Since $Z$ is a continuous stochastic process, it is progressively measurable with respect to its canonical filtration. Thus $Z$ is progressively measurable with respect to $(\hat{\mathcal{F}}_t)_t$. 
	
	We discuss now the identification of the limit for $j\rightarrow +\infty$, namely we prove that $Z$ is a solution to the SSME with multiplicative noise.
	Define the additional processes $M^{z^j}$, $M^j$ and $M$, given for all $j\in \mathbb{N}$ and for all $\phi\in H^1$ by
	\begin{align*}
	M^{z^j}_t:=\big \langle z^{\nu_j}_t-z^{\nu_j}_0 -\int_{0}^{t}z^{\nu_j}_r\times \partial^2_x z^{\nu_j}_r \dd r-(\sqrt{\nu_j}h+1)^2\int_{0}^{t}z^{\nu_j}_r \dd r+\nu_j \int_{0}^{t} z^{\nu_j}_r\times [z^{\nu_j}_r\times \partial_x^2 z^{\nu_j}_r]\dd r, \phi\big \rangle\, ,
	\end{align*}
	\begin{align*}
	M^{j}_t:= \big \langle Z^j_t-Z^j_0 -\int_{0}^{t}Z^j_r\times \partial^2_x Z^j_r \dd r-(\sqrt{\nu_j}h+1)^2\int_{0}^{t}Z^j_r \dd r+\nu_j \int_{0}^{t} Z^j_r\times [Z^j_r\times \partial_x^2 Z^j_r]\dd r\, , \phi\big \rangle\, ,
	\end{align*}
	\begin{align*}
	M_t:= \big \langle Z_t-Z_0 -\int_{0}^{t}Z_r\times \partial^2_x Z_r \dd r-\int_{0}^{t}Z_r \dd r \, , \phi\big \rangle\, .
	\end{align*}
	We need to prove that 
	\begin{itemize}
		\item[1.] $B$ is a Brownian motion with respect to $(\hat{\mathcal{F}}_t)_t$,
		\item[2.]  the following processes are $(\hat{\mathcal{F}}_t)_t$-martingales:
		\begin{align*}
		M\, , \quad G_{\cdot}:=\quad M^2-\int_{0}^{\cdot}\left|\int_DZ_r\times \phi \dd x\right|^2\dd r\, ,\quad \quad MB^k-\int_{0}^{\cdot}\int_D(Z_r\times \phi )^k\dd x\dd r\, ,
		\end{align*}
	\end{itemize}
	for $k=1,2,3$.
	The two above statements together imply that the quadratic variation of
	\begin{align*}
	M-\int_{0}^{\cdot} \langle  Z_r \times \dd B_r, \phi \rangle 
	\end{align*}
	is $0$ and thus that $\hat{\mathbb{P}}$-a.s.
	\begin{align*}
	\langle Z_t-Z_0,\phi\rangle -\int_{0}^{t}\langle Z_r\times \partial_x^2 Z_r ,\phi \rangle \dd r -\int_{0}^{t}\langle  Z_r,\phi \rangle \dd r=M_t=\int_{0}^{t} \langle  Z_r \times \dd B_r, \phi \rangle \, .
	\end{align*}
	We show that the quadratic variation of $M$ has the displayed form in the computations below.
	The first assertion follows from the fact that a Brownian motion is defined completely by its law: hence $B$ is a Brownian motion with respect to its canonical filtration. To show that  $B$ is a $(\hat{\mathcal{F}}_t)_t$-Brownian motion, we need to prove that the filtration $(\hat{\mathcal{F}}_t)_t$ is non-anticipative with respect to $B$. This follows as in Lemma \cite[Lemma 2.9.3]{martina_book}, since $B^j$ is a $(\hat{\mathcal{F}}^j_t)_t$-Brownian motion and thus the filtration $(\hat{\mathcal{F}}^j_t)_t$ is non-anticipative with respect to $B^j$. By \cite[Corollary 2.1.36]{martina_book},  $B$ is a $(\hat{\mathcal{F}}_t)_t$-Brownian motion.
	
	We prove that $G$ is an $(\hat{\mathcal{F}}_t)_t$-martingale. Since $M^{z^j}$ is a continuous and square integrable $(\mathcal{F}_t)_t$-martingale such that $M^{z^j}_0=0$, also the stochastic process
	\begin{align*}
	G^{z^j}:= M^{z^j, 2}-\int_{0}^{\cdot}\left|\int_D (\sqrt{\nu_j}h+1) z^{\nu^j}_r\times\phi\dd x\right|^2\dd r
	\end{align*}
	is an $(\mathcal{F}_t)_t$-martingale, indeed the quadratic variation of  $M^{z^j}$ is 
	\begin{align*}
	\int_{0}^{\cdot}\left|\int_D (\sqrt{\nu_j}h+1) z^{\nu^j}_r\times\phi\dd x\right|^2\dd r\, .
	\end{align*}
	From the property of the cross product $a\times b\cdot c= -a\times c\cdot b$, which holds for all $a,b,c \in \mathbb{R}^3$, we infer that 
	\begin{align*}
	\int_{0}^{t} \langle (\sqrt{\nu_j}h+1) z^{\nu^j}_r \times \dd W_r, \phi \rangle &=-\int_{0}^{t} \int_D (\sqrt{\nu_j}h+1) z^{\nu^j}_r \times \phi\dd x \cdot \dd W_r\\
	&=-\sum_{i=1}^{3}\int_{0}^{t} \left[\int_D (\sqrt{\nu_j}h+1) z^{\nu^j}_r \times \phi\dd x\right]^i  \dd W^{i}_r=:\sum_{i=1}^{3} I^i_t \, .
	\end{align*}
	From the independence of the components of the Brownian motion $W^i, W^k$, for $i\neq k\in \{1,2,3\}$, the stochastic processes $I^iI^k$ are $(\mathcal{F}_t)_t$-martingales for $i\neq k\in \{1,2,3\}$ such that the quadratic covariation $\langle\langle I^k, I^i \rangle  \rangle=0$ (see e.g. \cite[Lemma 8.1]{Baldi}). From \cite[Theorem 7.3]{Baldi}, 
	\begin{align*}
	(I^{i}_t)^2 -\int_{0}^{t}\left(\left[\int_D (\sqrt{\nu_j}h+1) z^{\nu^j}_r \times \phi\dd x\right]^{i}\right)^2 \dd r
	\end{align*}
	is an $(\mathcal{F}_t)_t$-martingale. Hence, 
	\begin{align*}
	\langle \langle \sum_{i=1}^{3} I^i_t  \rangle \rangle &= \sum_{i=1}^{3} \int_{0}^{t}\left(\left[\int_D (\sqrt{\nu_j}h+1) z^{\nu^j}_r \times \phi\dd x\right]^{i}\right)^2 \dd r= \int_{0}^{t}\left|\int_D (\sqrt{\nu_j}h+1) z^{\nu^j}_r \times \phi\dd x\right|^{2} \dd r\, .
	\end{align*}
	Fix $t \in [0,T]$ and let $\tilde{h}:\mathcal{Y}|_{[0,t]}\times \mathcal{Y}_W |_{[0,t]}\mapsto [0,1]$ be a continuous function. From the martingale property and for all $s\leq t$
	\begin{align*}
	\mathbb{E}\left[\tilde{h}(\rho_t z^{\nu^j},\rho_t W)G^{z^j}_t\right]=\mathbb{E}\left[\tilde{h}(\rho_s z^{\nu^j},\rho_s W)G^{z^j}_s\right]\, ,
	\end{align*}
	where $\rho_t: C([0,T]; \tilde{B})\rightarrow C([0,t]; \tilde{B})$ is a restriction operator for a Banach space $\tilde{B}$.
	We want to pass the martingale property from the approximation to the limit process $Z$. Define the map
	\begin{align*}
	G^{j}:= (M^{j})^2-\int_{0}^{\cdot}\left|\int_D (\sqrt{\nu_j}h+1) Z^j_r\times\phi\dd x\right|^2\dd r\, .
	\end{align*}
	From the equality in law due to the Skorokhod's theorem, it holds
	\begin{align*}
	\hat{\mathbb{E}}\left[\tilde{h}(\rho_t Z^{j},\rho_t W^j)G^j_t\right]=\hat{\mathbb{E}}\left[\tilde{h}(\rho_s Z^{j},\rho_sW^j)G^j_s\right]\, .
	\end{align*}
	To conclude the proof, we apply Vitali's convergence theorem to pass to the limit for $j\rightarrow +\infty$ under the sign of expectation. For $\phi\in H^1$, 
	\begin{align*}
	\mathbb{E}[|M^{z^j}_t|^4]\leq 2|D|^4\|\phi\|_{L^2}^4+\|\phi\|_{H^1}^4\mathbb{E}[\|\partial_x z^{\nu^j}\|^4_{L^4(L^2)}]<C\, ,
	\end{align*}
	which is uniformly bounded in $j$. Hence, by passing to the limit under the sign of expectation, $G$ is a $(\hat{\mathcal{F}}_t)_t$-martingale (the identification of the limit in the non-linearities is analogous as in the previous section on statistically stationary solutions to the SME).
	With analogous calculations, $M$ and $MB^k-\langle \langle MB^k \rangle \rangle$ are $(\hat{\mathcal{F}}_t)_t$-martingales. In conclusion, $(\hat{\Omega},\hat{\mathcal{F}}, (\hat{\mathcal{F}}_t)_t,\hat{\mathbb{P}},B, Z)$ is a martingale solution to the SSME.
	
	For part c., the proof coincides with the proof of Theorem \ref{th:solution_stochastic}: in particular, the conservation law $|\langle Z_t \rangle |^2$ is achieved by noticing that, since the stochastic integral in the SSME is space independent, the stochastic integral in the evolution of  $|\langle Z_t \rangle |^2$ 
	\begin{align*}
		\int_0^t \langle Z_r \rangle\cdot \langle Z_r \rangle\times \circ \dd W_r=0\, ,
	\end{align*}
	vanishes, since $a\times b\cdot a=0$ $\forall a,b,\in \mathbb{R}^3$. An analogous argument holds for the conservation law $\|\partial_x Z\|_{L^2}$. 
	Concerning Theorem \ref{th:existence_stationary} d., from Lemma \ref{lemma:condition_for_non_trivial} and the same proof of Theorem \ref{th:intro_solution_stochastic} a., the stationary solution  $Z$ such that $\hat{\mathbb{P}}(\|\partial_x Z_t\|_{L^2}^2>0)>0$ if and only if  $\partial_x h\neq 0$. 
	
	To prove the assertion in Theorem \ref{th:existence_stationary} e., we observe that the stationary spherical Brownian motion is a stationary solution to the stochastic LLG equation. 
	We construct a non-constant in space stationary solution by taking a space dependent $h$: from Theorem \eqref{th:solution_stochastic} d., this choice implies the existence of a stationary solution with non-trivial dynamics and thus different from the stationary Brownian motion. This shows that there are more stationary solutions to the SSME.
\end{proof}

\section{Relations with the statistically stationary solutions to other equations.}\label{sec:stat_sol_other_eq}
As discussed in R.~L.~Jerrard, D.~Smets \cite{jerrard_smets_T_1_S_2}, the Schr\"odinger map equation is related to the binormal curvature flow equation and to the cubic non-linear Schr\"odinger equation by means of certain geometric transformations, which we discuss in Section \ref{sec:other_equations}. 
We approach and discuss the following questions:
\begin{enumerate}
	\item Can we apply rigorously those transforms to the trajectories of the statistically stationary solutions?
	\item Are the transformed processes also statistically stationary solutions to the corresponding equations?
\end{enumerate}

\subsection{On relations of the Schr\"odinger map equation with the binormal curvature flow equation.}\label{sec:other_equations}
We follow the considerations in Section 1.1 in R.~L.~Jerrard, D.~Smets \cite{jerrard_smets_T_1_S_2}. From Theorem \ref{th:existence_statistical_solutions}, each trajectory $Z$ of the statistically stationary solutions to the SME belongs to $C([0,T];H^1)\cap L^\infty(0,T;H^2)$. Define for all $x\in \textcolor{black}{[0,2\pi]}$ and for all $t\in [0,T]$,
\begin{align*}
	\Gamma_t(x):=\int_{0}^{x} Z_t(z)\dd z \, .
\end{align*} 
By definition $Z_t(x)=\partial_x \Gamma_t (x)$, hence it holds
\begin{align*}
	\partial_t  \partial_x \Gamma_t (x)  =\partial_x[\partial_x \Gamma_t (x)\times \partial_x \partial_x \Gamma_t (x)] \, \implies \partial_x[	\partial_t  \Gamma_t (x)- \partial_x \Gamma_t (x)\times \partial^2_x  \Gamma_t (x)]=0\, .
\end{align*}
For any continuous map $t\mapsto g_t$ constant in space, by integrating in time
\begin{align*}
	\Gamma_t=\Gamma_0+\int_0^t \partial_x \Gamma_r\times \partial^2_x  \Gamma_r\dd r+\int_0^t g_r \dd r\, .
\end{align*}
By defining $\gamma_t(x):=\Gamma_t(x)-\int_{0}^{t}g_r\dd r$, each solution $\gamma$ fulfils
\begin{align}\label{eq:binormal_curvature_flow}
\gamma_{t}-\gamma_0=\int_{0}^{t} \partial_x \gamma_r \times \partial^2_x  \gamma_r \dd r\, ,
\end{align}
under the arc length condition $|\partial_z \gamma|^2=1$. 
Define the map $f:H^2\rightarrow H^3$ as 
\begin{align}\label{eq:map_SME}
f(v):=\int_{0}^{\cdot} v(x)\dd x\, .
\end{align}
We denote by $X_t(\hat{\omega}):=f(Z_t(\hat{\omega}))$ for $\hat{\mathbb{P}}$-a.e.~$\hat{\omega}\in \hat{\Omega}$, where $Z$ is a statistically stationary solution to \eqref{eq:schroedinger_eq}. 
We answer to the following question: is $X$ a statistically stationary solution to the binormal curvature flow \eqref{eq:binormal_curvature_flow}? The answer is affirmative and it is contained in Theorem \ref{th:statisticalBCF}. The proof is below.
\begin{proof} (proof of Theorem \ref{th:statisticalBCF})
 Since $Z(\hat{\omega})$ is in $C([0,T];H^1)\cap L^\infty(0,T;H^2)$ $\hat{\mathbb{P}}$-a.s., we can apply the transformation \eqref{eq:map_SME} $\hat{\mathbb{P}}$-a.s. and conclude that each trajectory $X(\hat{\omega})$ is a solution to the BCF. From a small adaptation of \cite[Corollary 2.11.9]{martina_book}, $X$ is stationary.
 
 We show the non-triviality results. Observe that $\partial_x X$ is never null, since $Z$ is sphere valued: in particular $|\partial_x X_t|^2=1$ for a.e.~$x\in D$ and for all $t\geq 0$.
 To prove (i), $\partial_x^2 X\neq 0$ if and only if $\partial_x Z\neq 0$. Then from Theorem \ref{th:intro_solution_stochastic} a.,  $\partial_x Z,\partial_x^2 Z= 0$ if and only if $\partial_x h=0$ and the claim follows. 
 
To prove (ii) under the assumption $\partial_x h\neq0$, observe that from (i) there exists a set of positive probability where $\partial_x X\times\partial_x^2 X\neq 0$. On the other hand, by taking expectation in the equation $\hat{\mathbb{E}}[\partial_x X\times\partial_x^2 X]=0$  and thus the map $\hat{\omega}\mapsto X(\hat{\omega})$ is not constant. 
We prove (ii) under the assumption that  $\partial_x h=0$: in this case, $Z=\langle Z\rangle$ and we can rewrite $X_t(x)=Z_t x$. From Lemma \ref{lemma:stochasticity_partial_x_h_0}, $\hat{\mathbb{E}}[X_t]=x\hat{\mathbb{E}}[Z_t]=0$ and, since $|\partial_x X_t|^2=1$, the process $X$ must be stochastic.

To prove (iii), observe that if $\partial_x h\neq 0$ then we have non-trivial dynamics in time for the map $t\mapsto X_t$ from (i) on a set of positive measure. To show the other implication, we prove that if $\partial_x h=0$, then $t\mapsto X_t$ is constant: if $\partial_x h=0$, then from Theorem \ref{th:intro_solution_stochastic} c. $t\mapsto Z_t$ is constant and, by definition, also $t\mapsto X_t$.
\end{proof}

\appendix 
\section{Useful inequalities.}\label{sec:appendix_a}
\subsection{On Poincaré and Poincaré-Wirtinger inequalities.}
We recall the classical Poincaré inequality 
\begin{theorem}\label{teo:poincare}(Poincaré inequality) Let $D\subset \mathbb{R}^n$ be an open bounded domain. Then for every $p\in [1,\infty)$ there exists a constant $C\equiv C(p,D,n)>0$, depending only on $p,D,n$, such that for all $v\in W^{1,p}_0(D)$
	\begin{align*}
	\|v\|_{L^p(D)}\leq C \|\nabla v\|_{L^p(D)}\, .
	\end{align*}
\end{theorem}
The application of the Poincaré inequality is restricted to functions which are null on the boundary in the sense of the trace. The Poincaré-Wirtinger inequality is an extension of the Poincaré's inequality to the whole $W^{1,p}(D)$ (see e.g. \cite[Corollary 5.4.1]{poincare_wirtinger}), which we recall.
\begin{theorem}\label{teo:Poin_wirt}
	Let $D\subset\mathbb{R}^n$ be an open connected bounded subset with $C^1$ boundary. For all $p\geq 2$ there exists a constant $C\equiv C(p,D,n)>0$, depending only on $p,D,n$,  such that for all $v\in W^{1,p}(D)$
	\begin{align*}
	\bigg\|v-\frac{1}{|D|}\int_{D}v(y)\dd y\bigg\|_{L^p(D)}\leq C \|\nabla v\|_{L^p(D)}\, .
	\end{align*}
\end{theorem}
\subsection{Interpolation inequalities on a one-dimensional bounded domain.}
Recall that from the Gagliardo-Nirenberg-Sobolev inequality on a bounded one-dimensional domain, the following inequality holds.
\begin{lemma}\label{lemma:interp_ladyz}
	Let $D\subset \mathbb{R}$ be a bounded connected open domain with $C^1$ boundary.
	For all $v\in W^{1,2}(D)$ there exists a constant $C>0$ depending on the domain such that it holds
	\begin{align*}
	\|v\|_{L^4(D)}\leq C \|v\|^{\frac{1}{4}}_{H^1(D)}\|v\|^{\frac{3}{4}}_{L^2(D)}\, .
	\end{align*}
\end{lemma}
\begin{proof}
	The assertion follows from e.g. \cite[Theorem 5.8]{Adams_Fournie} with $n=1$, $q=4$, $m=1$, $p=2$.
\end{proof}
We recall also the Agmon's interpolation inequality in one dimension.
\begin{lemma}
	Let $D\subset \mathbb{R}$ be an open bounded domain with $C^1$ boundary. There exists a constant $C>0$ such that for all $v\in H^1(D)$,
	\begin{align*}
	\|v\|_{L^\infty(D)}\leq C\|v\|_{H^1(D)}^{1/2}\|v\|_{L^2(D)}^{1/2}\, .
	\end{align*}
\end{lemma}

\section{The It\^o-Stratonovich correction term for the spherical noise.} \label{sec:appendix_b}
 We derive the It\^o-Stratonovich correction term for the equation fulfilled by spherical Brownian motion $y$, 
 \begin{align}\label{eq:SBM_appendix}
 y_t=y^0+\int_{0}^{t} y_r\times \circ \dd W_r\, ,\quad\quad \mathrm{for}\,\, t\geq 0 \, .
 \end{align}
 Recall that $W\equiv(W^1,W^2,W^3)$ is an $\mathbb{R}^3$-valued Brownian motion. Define 
 \begin{align*}
 \Gamma(v)\equiv (\Gamma^{i,j})_{i,j\in\{1,2,3\}}:=\sqrt{\nu} h v\times\cdot \, , \quad \mathrm{for} \quad v\equiv(v_1,v_2,v_3)\in \mathbb{R}^3\, ,
 \end{align*}
 which we rewrite for practical reasons as
 \begin{align}\label{eq:def_B}
 \Gamma(v)= \sqrt{\nu} h
 \left[ {\begin{array}{ccc}
 	0& -v_3  & v_2\\
 	v_3 & 0 &-v_1\\
 	-v_2	& 	v_1	&0 \\
 	\end{array} } \right]\, .
 \end{align}
 To pass from the Stratonovich integral to the It\^o integral in \eqref{eq:SBM_appendix}, we compute the quadratic covariation between $y$ and $W$. We fix $i\in \{1,2,3\}$ and we write
 \begin{align*}
 	\dd y^i= \sum_{k=1}^3\Gamma^{i,k}(y)\dd W^k\, , \quad \dd \Gamma^{i,k}(y)= \sum_{j=1}^3\frac{\partial \Gamma^{i,k}}{\partial y^j}(y)\dd y^j= \sum_{j,\ell=1}^3\frac{\partial \Gamma^{i,k}}{\partial y^j}(y) \Gamma^{j,\ell} (y)\dd W^\ell\, .
 \end{align*}
 By computing the covariation for $i\in \{1,2,3\}$, we obtain
 \begin{align*}
 	C^i(y):=\sum_{j,k=1}^3 \frac{\partial \Gamma^{i,k}}{\partial y^j}(y) \Gamma^{j,k} (y)\, .
 \end{align*}
 Then we can convert \eqref{eq:SBM_appendix} to its It\^o's formulation by
 \begin{align*}
 \int_{0}^{t} y_r\times \circ \dd W_r=\int_{0}^{t} y_r\times  \dd W_r-\frac{1}{2}\int_{0}^{t} C(y_r) \dd r\, .
 \end{align*}
 We compute the precise form of $C(y)$ in the following Lemma \ref{lemma:ito_stratonovich_correction}.
	\begin{lemma}\label{lemma:ito_stratonovich_correction}
	For all $t\geq 0$, it holds
	\begin{align*}
	\sqrt{\nu}\int_{0}^{t} h y_r\times \circ\dd W_r=\sqrt{\nu}\int_{0}^{t} h y_r\times \dd W_r-\nu \int_{0}^{t} h^2 y_r \dd r\, .
	\end{align*}
\end{lemma}
\begin{proof}
	
	Define the vector $C(v):=(C^1(v),C^2(v),C^3(v))$, where
	\begin{align*}
	C^i(v):=\sum_{j=1}^{3}\sum_{k=1}^{3} \frac{\partial \Gamma^{i,k}(v)}{\partial v^j} \Gamma^{j,k}(v)=\sum_{j=1}^{3} C^{i,j}(v)\, 
	\end{align*}
	for $i=1,2,3$.
	We start by computing $C^1$ and we note that
	\begin{align*}
	&C^{1,1}(v)=\sum_{k=1}^{3} \frac{\partial \Gamma^{1,k}(v)}{\partial v^1} \Gamma^{1,k}(v)= 0  \, ,\\
	&C^{1,2}(v)=\sum_{k=1}^{3}  \frac{\partial \Gamma^{1,k}(v)}{\partial v^2} \Gamma^{2,k}(v)=  -\nu h^2 v_1 \, ,\\
	&C^{1,3}(v)=\sum_{k=1}^{3}  \frac{\partial \Gamma^{1,k}(v)}{\partial v^3} \Gamma^{3,k}(v)=  -\nu h^2 v_1 \, .
	\end{align*}
	Hence $C^1(v)=-2\nu h^2 v_1$. With analogous computations, we can conclude that $C(v)=-2\nu h^2 v$. Hence, since the It\^o-Stratonovich correction term is $C(v)/2$. This applies to the case $v=y$.
\end{proof}
\begin{lemma}\label{lemma:traccia}For all $v\equiv(v_1,v_2,v_3)\in \mathbb{R}^3$, where $\Gamma(v)$ is defined in \eqref{eq:def_B}, 
\begin{align*}
\mathrm{tr}(\Gamma(v) I \Gamma(v)^T)=2 \nu h^2 (v_1^2+v_2^2+v_3^2)\, .
\end{align*}
\end{lemma}
\begin{proof}
In the notations of the proof in Lemma \ref{lemma:ito_stratonovich_correction}, 
\begin{align*}
	\Gamma^{i,i}(v)=\nu h^2 (v_j^2+v_k^2)\, , 
\end{align*}
for all $i,j,k\in \{1,2,3\}$, $i\neq j\neq k$.
\end{proof}
\begin{lemma}\label{lemma:ito_formula_square}
Let $m\in L^\infty(0,T;H^1)\cap L^2(0,T;H^2)\cap C([0,T];H^1)$ be the unique solution to the stochastic LLG equation \eqref{LLG}, then for all $ 0\leq s\leq t$
\begin{align*}
\| \partial_x m_t\|_{L^2}^2-\| \partial_x m_s\|_{L^2}^2= -2\nu \int_{s}^{t}\|m_r\times \partial_x^2 m_r\|^2_{L^2} \dd r&+2\nu (t-s)\|\partial_x h\|^2_{L^2}\\
&+2\int_{s}^{t} \left(\partial_x m_r,\sqrt{\nu} \partial_xh m_r\times \dd W_r\right)\, .
\end{align*}
\end{lemma}
\begin{proof}
Recall that for every initial condition $m^0\in H^1$, the unique solution $m$ belongs to $C([0,T];H^1)\cap L^2(0,T;H^2)$.
From Lemma \ref{lemma:ito_stratonovich_correction}, the LLG with It\^o noise reads
\begin{align*}
m_t-m_s=\int_{s}^{t} A^\nu(m_r) \dd r+ \int_{s}^{t} \sqrt{\nu} h m_r\times \dd W_r\, ,
\end{align*}
where $A^\nu(m_r) = -m_r\times \partial_x^2 m_r+\nu m_r\times [m_r\times \partial_x^2 m_r]-\nu  h^2 m_r\, .$
We apply the multidimensional It\^o's formula to $F(\partial_x m):=\|\partial_x m\|^2_{L^2}$. Since $F'(\partial_x  m)= 2 \partial_x  m\in L^2$ and $F''(\partial_x m)=2I\in \mathbb{R}^{3\otimes 3}$, we have
\begin{align*}
	\| \partial_x m_t\|_{L^2}^2-\| \partial_x m_s\|_{L^2}^2&=2\int_{s}^{t} \left(\partial_x A^\nu(m_r) , \partial_x m_r\right) \dd r\\
	&\quad+\frac{1}{2}\int_{0}^{T}\int_D \left(\mathrm{tr}\left(\Gamma(\partial_x(\sqrt{\nu}hm_r))^T2 I \Gamma(\partial_x(\sqrt{\nu}hm_r))\right)\right) \dd x \dd r\\
	&\quad+2\int_{s}^{t} \left(\partial_x m_r,\sqrt{\nu} \partial_x(h m_r)\times \dd W_r\right)\, .
\end{align*}
The drift can be rephrased as
\begin{align*}
\left(\partial_x A^\nu(m_r) , \partial_x m_r\right)&=-\nu \|m_r\times \partial_x^2 m_r\|^2_{L^2}-\nu \left(2h\partial_x h m_r+h^2\partial_x m, \partial_x m\right)\\
&= -\nu \|m_r\times \partial_x^2 m_r\|^2_{L^2}-\nu \|h\partial_x m\|_{L^2}^2\, ,
\end{align*}
where in the last equality we use that $\partial_x m_r\cdot m_r=0$ for a.e. $(t,x)\in [0,T]\times D$ and $\mathbb{P}$-a.s.
From Lemma \ref{lemma:traccia} 
\begin{align*}
\mathrm{tr}(\Gamma(\sqrt{\nu}hm_r) I \Gamma(\sqrt{\nu}hm_r)^T)=2  |\partial_x(\sqrt{\nu}hm_r)|^2&=2\nu(\partial_x h m_r+h\partial_x m_r)\cdot (\partial_x h m_r+h\partial_x m_r)\\
&=2\nu[|\partial_x h|^2 +h^2|\partial_x m|^2]\, ,
\end{align*}
where in the last equality we use that $m_r\in \mathbb{S}^2$ and that $\partial_x m_r\cdot m_r=0$ for a.e. $(t,x)\in [0,T]\times D$ and $\mathbb{P}$-a.s.
From the properties of the cross product, $\partial_x m\cdot \partial_x m \times \dd W_r=0$. In conclusion, the evolution of  $\|\partial_x m\|^2_{L^2}$ is described by
\begin{align*}
\| \partial_x m_t\|_{L^2}^2-\| \partial_x m_s\|_{L^2}^2= -2\nu \int_{s}^{t}\|m_r\times \partial_x^2 m_r\|^2_{L^2} \dd r&+2\nu\int_{s}^{t}\|\partial_x h\|^2_{L^2}\dd r\\
&+2\int_{s}^{t} \left(\partial_x m_r,\sqrt{\nu} \partial_xh m_r\times \dd W_r\right)\, .
\end{align*}
Since $h$ is constant in time, the assertion follows.
\end{proof}

\section{Embeddings}\label{sec:appendix_c}
We introduce some useful known results. Lemma \ref{lemma:tornstein} is contained in F.~Flandoli, D.~Gatarek \cite[Theorem 2.1]{flandoli_gatarek}.
\begin{lemma} \label{lemma:tornstein}
	Let $H_0\subset H \subset H_1$ be Banach spaces, where $H_0, H_1$ are reflexive and $H_0$ is compactly embedded into $H$. Let $p\in (1,\infty)$ and $\alpha\in (0,1)$ be given. Let $\mathcal{X}:= L^p(0,T; H_0)\cap W^{\alpha,p}(0,T;H_1)$ endowed with the natural norm. Then embedding of $\mathcal{X}$ into $L^p(0,T;H)$ is \textcolor{black}{compact}.
\end{lemma}
A further useful embedding theorem, see e.g. \cite[Theorem 1.8.5]{martina_book}.
\begin{lemma}\label{th:embedding_c_w_appendix}
Let $\alpha > 0$, $1, p <\infty$ and $\ell \in \mathbb{R}$. Then the embedding of 
\begin{align*}
	L^\infty (0,T;L^p)\cap C^\alpha([0,T]; W^{\ell, 2})
\end{align*}
into 
\begin{align*}
C_{\mathrm{w}}([0,T]; L^p)
\end{align*}
is sequentially compact.
\end{lemma}

\section{Stationarity} \label{sec:appendix_d}
In the paper, we make use of  \cite[Lemma A.5]{breit_fereisil_hofmanova_masloski}: the result gives sufficient conditions for the stationarity of a limit of stationary processes.
\begin{proposition}\label{pro:sufficient_condition_stationary}
	Let $k\in \mathbb{N}_0$, $p\in [1,+\infty)$ and $(X^n)_n$ be a sequence of $W^{k,p}$-valued stochastic processes which are stationary on $W^{k,p}$. Assume that the following two conditions hold
	\begin{enumerate}
		\item (Uniform bound) For all $T>0$
		\begin{align*}
		\sup_{n\in \mathbb{N}}\mathbb{E}\left[\sup_{t\in [0,T]}\|X^n_t\|^2_{W^{k,p}}\right]<+\infty\, .
		\end{align*}
		\item(Weak convergence) It holds $\mathbb{P}$-a.s. that
		\begin{align*}
		X^n\longrightarrow X \quad \mathrm{in} \quad C_{\mathrm{w}}([0,+\infty); W^{k,p})\, .
		\end{align*}
	\end{enumerate}
	Then $X$ is stationary on $W^{k,p}$.
\end{proposition}

\bibliographystyle{plain}

\begin{thebibliography}{9}

\bibitem{Adams_Fournie}
R.~A.~Adams, J.~J.~F.~Fournier.
\textit{Sobolev spaces.} 
Second edition, 2003.

\bibitem{poincare_wirtinger}
H.~Attouch, G.~Buttazzo, G.~Michaille.
\textit{Variational analysis in Sobolev and BV spaces.}
MOS-SIAM Series on Optimization. Second Edition, 2014. 

\bibitem{Baldi}
P.~Baldi. Stochastic Calculus. An introduction through theory and exercises. Springer, Universitex, 2017.

\bibitem{banica_vega_1}
V. Banica and L. Vega.
Stability of the self similar dynamics of a
vortex filament, Arch. Ration.
Mech. Anal. 210, 673–712, 2013.

\bibitem{banica_vega_2}
V. Banica and L. Vega.
Evolution of polygonal lines by the binormal
flow, Ann. PDE, 6,
Paper No. 6, 53 pp., 2020.

\bibitem{banica_vega_3}
V. Banica and L. Vega, On the energy of critical solutions of the
binormal flow, Comm. PDE, 45,
820–845, 2020.

\bibitem{banica_vega_4}
Banica, V., Vega, L. Riemann's Non-differentiable Function and the
Binormal Curvature Flow. Arch. Ration. Mech. Anal. 244, 501–540,
2022.

\bibitem{bedrossiam_coti_zelati_holtz}
Jacob Bedrossian, Michele Coti Zelati, and Nathan Glatt-Holtz. Invariant measures for passive scalars in the small noise inviscid limit. Communications in Mathematical Physics, 348:101–127, 2016.

\bibitem{Bourgain}
J.~Bourgain. Invariant measures for the 2D-defocusing nonlinear Schr\"odinger
equation.
\textit{Comm. Math. Phys.} 176: 421-455, 1996.

\bibitem{bourgain_94}
J. Bourgain. 
Periodic nonlinear Schr\"odinger equation and invariant measures. Comm. Math. Phys., 166(1):1–26, 1994.

\bibitem{breit_fereisil_hofmanova_masloski}
D.~Breit, E.~Feireisl,  M.~Hofmanová, B.~Maslowski.
Stationary solutions to the compressible Navier–Stokes system driven by stochastic forces. Probab. Theory Relat. Fields 174, 981–1032, 2019. 

\bibitem{martina_book}
D.~Breit, E.~Fereisl, M.~Hofmanová.
Stochastically forced compressible fluid flows.
Series in Applied and Numerical Mathematics, De Gruyter, 2018.

\bibitem{brzezniak_LDP}
Z.~Brze\'zniak, B.~Goldys, T.~Jegaraj.
Large deviations and transitions between equilibria for stochastic Landau-Lifshitz-Gilbert equation.
\textit{Archive for Rational Mechanics and Analysis.} 1--62, 2017.

\bibitem{da_rios}
L.~S.~Da Rios, On the motion of an unbounded fluid with a vortex filament of any shape, Rend. Circ. Mat. Palermo 22, 117–135, 1906.



\bibitem{fahim_hausenblas}
K.~Fahim, E.~Hausenblas, D.~Mukherjee. 
Wong–Zakai Approximation for Landau–Lifshitz–Gilbert Equation Driven by Geometric Rough Path. Applied mathematics and optimization, \textbf{46}, 1685-1730, 2021.


\bibitem{ferrario_zanella_stat_sol_NSE}
B.~Ferrario, M.~Zanella.
Stationary solutions for the nonlinear Schr\"odinger equation.  arXiv:2305.10393, 2023.

\bibitem{flandoli_gatarek}
F.~Flandoli, D.~Gatarek.
Martingale and stationary solutions for stochastic Navier-Stokes equations. Probab. Th. Rel. Fields 102, 367–391, 1995.

\bibitem{flodes_sy}
J.~F\"odes and M.~Sy. Invariant measures and global well posedness for SQG equation. arXiv:2002.09555, 2020.

\bibitem{FrizHairer}
P.~Friz, M.~Hairer.
\textit{A course on rough paths: with an introduction to regularity structures}.
Universitext, Springer, 2014.

\bibitem{stochastic_LL}
B.~Guo, X.~Pu.
Stochastic Landau-Lifschitz equation. Differential Integral Equations 22 (3/4), 251 - 274, 2009.

\bibitem{LLG_inv_measure}
E.~Gussetti.
On ergodic invariant measures for the stochastic Landau-Lifschitz-Gilbert equation in 1D. Arxiv preprint, 2022.

\bibitem{CLT}
E.~Gussetti.
Pathwise central limit theorem and moderate deviations via rough paths for SPDEs with multiplicative noise.
arXiv preprint arXiv:2307.10965, 2023.

\bibitem{LLG1D}
E.~Gussetti, A.~Hocquet.
A pathwise stochastic Landau-Lifshitz-Gilbert equation with application to large deviations. Journal of Functional Analysis, 285 (9), 110094, 2023.

\bibitem{hashimoto}
H.~Hashimoto.
A soliton on a vortex filament, J. Fluid Mech., 51, no. 3, pp. 477–485, 1972.

\bibitem{holtz_sverak}
N.~Glatt-Holtz, V.~\v{S}verák, and V.~Vicol. On inviscid limits for the stochastic Navier–Stokes equations and related models. Arch. Ration. Mech. Anal., 217:619–649, 2015.

\bibitem{jerrard_seis}
R. L. Jerrard and C. Seis. On the vortex filament conjecture for Euler flows, Arch. Ration. Mech. Anal. 224, 135–172, 2017.

\bibitem{jerrard_smets_T_1_S_2}
R.~L.~Jerrard, D.~Smets. On Schr\"odinger maps from T1 to S2. Ann. Sci. ENS, vol. 45, no. 4, 635-678, 2012.

\bibitem{kuksin_2003}
S.~Kuksin.
The Eulerian Limit for 2D Statistical Hydrodynamics.
Journal of Statistical Physics, Vol. 115, Nos. 1/2, 2004.

\bibitem{kuksin_2008}
S.~Kuksin.
On Distribution of Energy and Vorticity for Solutions of 2D Navier-Stokes Equation with Small Viscosity. Commun. Math. Phys. 284, 407, 2008. 

\bibitem{kuksin_KdV_2008}
S.~Kuksin, A.~Piatnitski. Khasminskii–Whitham averaging for randomly perturbed KdV equation. J. Math. Pures Appl. 89, 400-428, 2008.

\bibitem{kuksin_KdV_2010}
S.~Kuksin. Damped-driven KdV and effective equations for long-time behaviour of its solutions. Geom. Funct. Anal. 20, 1431-1463, 2010.

\bibitem{kuksin_shirikyan_2004}
S.~Kuksin, A.~Shirikyan.
Randomly forced CGL equation: stationary measures and the inviscid limit. J. Phys. A: Math. Gen. \textbf{37} 3805, 2004.

\bibitem{latocca}
M.~Latocca. Construction of high regularity invariant measures for the 2D Euler equations and remarks on the growth of the solutions. Communications in Partial Differential Equations, 48 (1), 22-53, 2023.

\bibitem{lebowitz_rose_speer}
J. L. Lebowitz, H. A. Rose, and E. R. Speer. Statistical mechanics of the nonlinear Schr\"odinger equation. J. Statist. Phys., 50(3-4):657–687, 1988.

\bibitem{protter}
P.~E.~Protter. 
Stochastic integration and differential equations, 2005.

\bibitem{Shirikyan_local_times}
A.~Shirikyan.
Local times for solutions of the complex Ginzburg–Landau equation and
the inviscid limit. Journal of Mathematical Analysis and Applications.

\bibitem{sy_xu_2021}
M.~Sy, X.~Yu.
Almost sure global well-posedness for the energy supercritical NLS on the unit ball of $\mathbb{R}^3$. Preprint arXiv:2007.00766, 2020.

\bibitem{sy_19}
M.~Sy.
Almost sure global well-posedness for the energy supercritical Schr\"odinger equations.  J. Math. Pures. Appl. (154) 108--145, 2019.

\bibitem{sy_xu_2022}
M.~Sy, X.~Yu.
Global well-posedness for the cubic fractional NLS on the unit disk.
Nonlinearity (35), 2020, 2022.

\bibitem{sy_xu_2021_2}
M.~Sy, X.~Yu.
Global well-posedness and long-time behavior of the fractional NLS.
Stochastics and Partial Differential Equations: Analysis and Computations (10), 1261–1317, 2021.

	
	
\end{thebibliography}

\end{document}